\documentclass[12pt]{amsart}
\usepackage{a4wide,enumerate,color}
\usepackage{xcolor}
\usepackage{bm} 
\usepackage{scrextend} 
\allowdisplaybreaks
\usepackage{enumitem}
\let\pa\partial  
\let\na\nabla  
\let\eps\varepsilon  
  
\newcommand{\R}{{\mathbb R}} 
\newcommand{\diver}{\operatorname{div}}  
\newcommand{\E}{{\mathcal E}}

\newcommand{\brho}{\bm{\rho}}

\newtheorem{theorem}{Theorem}   
\newtheorem{lemma}[theorem]{Lemma}   
\newtheorem{proposition}[theorem]{Proposition}   
\newtheorem{remark}[theorem]{Remark}   
  
\newtheorem{definition}{Definition}

\begin{document}  

\title[Euler flows for multicomponent systems]{High-friction limits of
Euler flows for multicomponent systems}

\author[X. Huo]{Xiaokai Huo}
\address{Computer, Electrical and Mathematical Science and Engineering Division,
King Abdullah University of Science and Technology (KAUST),
Thuwal 23955-6900, Saudi Arabia}
\email{xiaokai.huo@kaust.edu.sa}

\author[A. J\"ungel]{Ansgar J\"ungel}
\address{Institute for Analysis and Scientific Computing, Vienna University of  
	Technology, Wiedner Hauptstra\ss e 8--10, 1040 Wien, Austria}
\email{juengel@tuwien.ac.at} 

\author[A. Tzavaras]{Athanasios E. Tzavaras}
\address{Computer, Electrical and Mathematical Science and Engineering Division,
King Abdullah University of Science and Technology (KAUST),
Thuwal 23955-6900, Saudi Arabia}
\email{athanasios.tzavaras@kaust.edu.sa}

\date{\today}

\thanks{
Part of this manuscript was written during the stay of the second author at the 
King Abdullah University of Science and Technology (KAUST). He thanks KAUST for
the hospitality and support during his stay. 
Furthermore, he acknowledges partial support from   
the Austrian Science Fund (FWF), grants F65, P27352, P30000, and W1245} 

\begin{abstract}
The high-friction limit in Euler-Korteweg equations for fluid mixtures is
analyzed. The convergence of the solutions towards the zeroth-order limiting system
and the first-order correction is shown, assuming suitable uniform bounds.
Three results are proved: The first-order correction system is shown to be of
Maxwell-Stefan type and its diffusive part is parabolic in the sense of
Petrovskii. The high-friction limit towards the first-order Chapman-Enskog 
approximate system is proved in the weak-strong solution context for general 
Euler-Korteweg systems. Finally, the limit towards the
zeroth-order system is shown for smooth solutions in the isentropic case
and for weak-strong solutions in the Euler-Korteweg case.
These results include the case of constant capillarities and multicomponent
quantum hydrodynamic models.
\end{abstract}

 

\maketitle


\section{Introduction}

Multicomponent flows appear in many applications including sedimentation,
dialysis, electrolysis, and ion transport \cite{WeKr99}. These flows may
be described by Euler or Euler-Korteweg 
equations for the various species, coupled through
interaction forces proportional to the difference of the partial velocities. 
The equations can be simplified when the interaction is strong, leading in
the zeroth-order limit to the Euler equations for the partial particle densities 
and the common velocity
and in the first-order correction to diffusive systems of Maxwell-Stefan type
coupled with the momentum balance equation for the barycentric velocity. 
While such relaxation and high-friction
limits are widely explored in mono-species situations, there are no results
for multicomponent Euler-Korteweg flows. The aim of this paper is to
compute the Chapman-Enskog expansion and 
to justify the expansion via a relative entropy approach, extending
results for the mono-species case to fluid mixtures \cite{GLT17,LaTz13,LaTz17}. 

We consider the following Euler-Korteweg equations for multicomponent fluids,
\begin{align}
  \pa_t\rho_i + \diver(\rho_i v_i) &= 0, \label{1.eq1} \\
  \pa_t(\rho_i v_i) + \diver(\rho_i v_i\otimes v_i)
	&= -\rho_i\na\frac{\delta\E}{\delta\rho_i}(\brho) 
	- \frac{1}{\eps}\sum_{j=1}^nb_{ij}\rho_i\rho_j(v_i-v_j), \label{1.eq2}
\end{align}
where $i=1,\ldots,n$, $x\in\R^3$, $t>0$, and $\brho =\brho(x,t)
= (\rho_1,\ldots,\rho_n)(x,t)$. The initial conditions are
$$
  \rho_i(\cdot,0)=\rho_{i}^0, \quad v_i(\cdot,0)=v_{i}^0 \quad\mbox{in }\R^3,\ 
	i=1,\ldots,n.
$$
The variables $\rho_i$ are the partial densities and $v_i$ the partial
velocities. The parameters $b_{ij}\ge 0$ model the interaction of the
$i$th and $j$th components with a strength that is measured by $\eps>0$.
Model \eqref{1.eq1}-\eqref{1.eq2} belongs to the general realm of multicomponent
fluid mixtures whose thermodynamical structure has been extensively analyzed;
see, e.g., \cite{BoDr15, MuRu93, RuSi09} and references therein.
On the other hand, we adopt the mathematical structure espoused in \cite{GLT17}, 
in that the dynamics of the flow is determined by the functional $\E(\brho)$ 
of potential energy, with $\delta\E/\delta\rho_i$ standing for the variational 
derivatives with respect to the partial densities $\rho_i$.
Several isothermal models fit into this framework.
In this work, we consider energies of the form
$$
  \E(\brho) = \int_{\R^3}\sum_{i=1}^n F_i(\rho_i,\na\rho_i)dx.
$$

For instance, when $F_i=h_i(\rho_i)$ for some (convex) function $h_i$ we obtain the 
equations of multicomponent system of gas dynamics with friction. When 
\begin{equation}\label{1.korte}
  F_i = h_i(\rho_i) + \frac12\kappa_i(\rho_i)|\na\rho_i|^2 \, ,
\end{equation}
we obtain  the multicomponent Euler-Korteweg system
$$
	\pa_t(\rho_iv_i) + \diver(\rho_iv_i\otimes v_i)
	=   \diver S_i[\rho_i]
  - \frac{1}{\eps}\sum_{j=1}^n b_{ij}\rho_i\rho_j(v_i-v_j),
$$
where
$$
  S_i[\rho_i] 
	:= \bigg(-p_i(\rho_i)-\frac12\big(\rho_i\kappa'_i(\rho_i)+\kappa_i(\rho_i)\big)
	|\na\rho_i|^2+\diver(\rho_i\kappa_i(\rho_i)\na\rho_i)
	\bigg){\mathbb I}
	- \kappa_i(\rho_i)\na\rho_i\otimes\na\rho_i
$$
is the stress tensor associated with the $i$th component and 
$p_i(\rho_i) = \rho_i h_i'(\rho_i)-h_i(\rho_i)$ is the partial pressure.
A special case is the selection $\kappa_i(\rho_i)= k_i/(4 \rho_i)$ with 
$k_i =\mbox{const.}$, which yields the multicomponent quantum hydrodynamic 
system with friction,
$$
	\pa_t(\rho_iv_i) + \diver(\rho_iv_i\otimes v_i) + \nabla p_i(\rho_i)
	=  \frac{1}{2} k_i \rho_i\na\bigg(\frac{\Delta\sqrt{\rho_i}}{\sqrt{\rho_i}}\bigg)
  - \frac{1}{\eps}\sum_{j=1}^n b_{ij}\rho_i\rho_j(v_i-v_j),
$$
used to describe quantum effects in semiconductors \cite{Jue09} or 
multicomponent quantum plasmas \cite{MSS08}.
The dependence of $F_i$ on the density (and its gradient) of the $i$th component
is crucial; the general case leads to mixed terms like $\pa F_i/\pa\rho_j$
that we cannot control.

The interaction term (the last term in \eqref{1.eq2}) has an alignment effect on the
partial velocities, and we expect that all partial velocities are the same
in the high-friction limit $\eps\to 0$, leading to the zeroth-order limit system
\begin{equation}\label{1.lim}
  \pa_t\bar\rho_i + \diver(\bar\rho_i\bar v) = 0, \quad
	\pa_t(\bar\rho\bar v) + \diver(\bar\rho\bar v\otimes\bar v)
	= -\sum_{i=1}^n\bar\rho_i\na\frac{\delta\E}{\delta\rho_i}(\bar\brho)
\end{equation}
for $i=1,\ldots,n$, where $\bar\brho = (\bar\rho_1,\ldots,\bar\rho_n)$, 
while $\bar\rho=\sum_{i=1}^n\bar\rho_i$ stands for the total density.
In the first-order correction, the solution  
$(\brho^\eps,\bm{v}^\eps)=(\rho_i^\eps,v_i^\eps)_{i=1,\ldots,n}$
to the hyperbolic relaxation system \eqref{1.eq1}-\eqref{1.eq2}
is expected to be close to the hyperbolic-diffusive system 
\begin{align}
  \pa_t\widehat\rho_i^\eps + \diver(\widehat\rho_i^\eps\widehat v^\eps)
	&= \eps\diver\sum_{j=1}^n D_{ij}^\eps(\widehat\brho^\eps)
	\na\frac{\delta\E}{\delta\rho_j}(\widehat\brho^\eps), \label{1.CE1} \\
	\pa_t(\widehat\rho^\eps\widehat v^\eps) + \diver(\widehat\rho^\eps
	\widehat v^\eps\otimes\widehat v^\eps) &= -\sum_{i=1}^n\widehat\rho_i^\eps
	\na\frac{\delta\E}{\delta\rho_i}(\widehat\brho^\eps) \label{1.CE2}
\end{align}
for $i=1,\ldots,n$, where $\widehat\brho^\eps = ( \widehat\rho^\eps_1,\ldots, 
\widehat\rho^\eps_n)$ and $\widehat\rho^\eps=\sum_{i=1}^n\widehat\rho_i^\eps$.
When the barycentric velocity $\widehat v^\eps$ vanishes, we recover the
Maxwell-Stefan equations analyzed in, e.g., \cite{Bot11,BGV17,JuSt13}.

Before stating our main results, we review the state of the art.
The structure of relaxation systems and their relaxation limits were first explored 
for examples \cite{CaPa79} and later for general systems 
\cite{Bou04,CLL94,Tza05,Yon04}. We call the limit $\eps\to 0$ a relaxation limit
if the time scale is of order $O(1/\eps)$.
Rigorous relaxation limits in the mono-species Euler equations 
towards the heat or porous-medium equation were proved, 
using energy estimates \cite{CoGo07}, the relative entropy approach \cite{LaTz13}, 
or convergence in Besov spaces \cite{XuKa14}. The relaxation limit
in non-isentropic flows was analyzed in, e.g., \cite{Wu16,XuYo09}.

When the potential energy $\E$ also depends on the gradient of the particle 
density, system \eqref{1.eq1}-\eqref{1.eq2} is of Euler-Korteweg type.
The relaxation (or high-friction) limit in these equations for 
single species was studied in
\cite{LaTz17} for monotone pressures (i.e.\ convex energies) and in 
\cite{GiTz17} for non-monotone pressures. 
Giesselmann et al.\ \cite{GLT17} proved stability theorems for the Euler-Korteweg system 
between a weak and a strong solution and for the Navier-Stokes-Korteweg system.

All these results concern the mono-species case. Relaxation limits in
multi-species systems were proved in the Euler-Poisson equations
for electrons and positively charged ions in plasmas or semiconductors
\cite{JuPe99}.  
At the zeroth order, such a limit leads to equations \eqref{1.lim}.
First-order corrections can be derived by a Chapman-Enskog expansion or
Maxwell-iteration technique. This was done in the Euler system with temperature
\cite{RuSi09}, leading to equations for multitemperature mixtures in
nonequilibrium thermodynamics. The Chapman-Enkog expansion was validated in
\cite{YaYo15,YYZ15} in the isentropic case, proving an error estimate for the
difference of the solutions of equations \eqref{1.eq1}-\eqref{1.eq2} and \eqref{1.CE1}-\eqref{1.CE2}. Another validation was recently presented
by Boudin et al.\ \cite{BGP18} by applying the formalism of Chen, Levermore, 
and Liu \cite{CLL94}. However, no results seem to be available in the literature
for high-friction limits in Euler-Korteweg systems. 

In this paper, we prove the convergence of solutions to \eqref{1.eq1}-\eqref{1.eq2}
towards the limit system \eqref{1.lim} and the first-order correction system
\eqref{1.CE1}-\eqref{1.CE2}. The main results can be sketched as follows:

\begin{labeling}{1a.}
\item[1.] We compute the Chapman-Enskog expansion leading to 
\eqref{1.CE1}-\eqref{1.CE2} and show that \eqref{1.CE1} has a gradient-flow structure
(Lemma \ref{lem.diff}). 
Moreover, when the barycentric velocity $\widehat v^\eps$ vanishes, the system 
is proved to be parabolic in the sense of Petrovskii (Lemma \ref{lem.Estar}).

\item[2.] \label{re2}
Assume that the functional \eqref{1.korte} satisfies some convexity conditions.
For weak solutions to the relaxation system \eqref{1.eq1}-\eqref{1.eq2}
and strong solutions to the approximate system \eqref{1.CE1}-\eqref{1.CE2}
with uniform bounds on the velocities, assuming that the difference of the
initial data is of order $O(\eps^2)$, we prove that 
\begin{align*}
  	\chi(t) :&= \int_{\R^3}\sum_{i=1}^n
	\bigg(\frac12\rho^\eps_i|v_i^\eps-\widehat v_i^\eps|^2
	+ (\rho_i^\eps-\widehat\rho_i^\eps)^2 
	+ \frac{1}{2\kappa_i(\rho_i^\eps)}|\kappa_i(\rho_i^\eps)\nabla \rho_i^\eps 
	- \kappa_i(\widehat \rho_i^\eps) \nabla\widehat \rho_i^\eps|^2 
	\bigg)(t)dx \\&\le C(\chi(0) + \eps^2)
\end{align*}  
uniformly in $t\in(0,T)$ for some constant $C>0$ independent of $\eps$, 
see Theorem \ref{thm.convCE}.

\item[3a.] Isentropic case: Smooth solutions to \eqref{1.eq1}-\eqref{1.eq2}
converge towards a smooth solution to the limit system \eqref{1.lim} in the sense
$$
  \sup_{0<t<T}\int_{\R^3}\sum_{i=1}^n\big((\rho_i^\eps-\bar\rho_i)^2
	+ |v_i^\eps-\bar v|^2\big)dx \to 0 \quad\mbox{as }\eps\to 0,
$$
if the initial relative entropy converges to zero; see Theorem \ref{thm.isen}.

\item[3b.] \label{re3a} Euler-Korteweg case with functional \eqref{1.korte}: 
Weak solutions to \eqref{1.eq1}-\eqref{1.eq2}
converge towards a strong solution to the limit system \eqref{1.lim} in the sense
\begin{align*}
    \chi(t) :&= \int_{\R^3}\sum_{i=1}^n
	\bigg(\frac12\rho^\eps_i|v_i^\eps-\bar v|^2
	+ (\rho_i^\eps-\bar\rho_i)^2 
	+ \frac{1}{2\kappa_i(\rho_i^\eps)}|\kappa_i(\rho_i^\eps)\nabla \rho_i^\eps 
	- \kappa_i(\bar \rho_i) \nabla \bar\rho_i|^2\bigg)(t)dx \\
	&\le C(\chi(0) + \eps)
\end{align*} 
uniformly in $t\in(0,T)$ for some constant $C>0$ independent of $\eps$; 
see Theorem \ref{thm.korte}. 
\end{labeling}

For these results, we need that the functions $\rho_i^\eps$ are uniformly bounded 
away from vacuum as well as $h_i$ and $-1/\kappa_i$ are convex. 
The case of the multicomponent quantum hydrodynamic system and the system with 
constant capillarities are included.

The idea of the proofs is to estimate the relative entropy between two solutions
$$
  \E_{\rm tot}(\brho,\bm{m}|\widehat\brho,\widehat{\bm{m}})(t) =
	\int_{\R^3}\sum_{i=1}^n\bigg(F_i(\rho_i,\na\rho_i|\widehat\rho_i,
	\na\widehat\rho_i) + \frac12\rho_i|v_i-\widehat v_i|^2\bigg)(t)dx,
$$
where $\bm{m} = (m_1,\ldots,m_n)$ with $m_i=\rho_i v_i$, 
$\widehat{\bm{m}} = (\widehat m_1,\ldots,\widehat m_n)$ with 
$\widehat m_i=\widehat \rho_i \widehat v_i$, and 
$F_i(\rho_i,\na\rho_i|\widehat\rho_i,\na\widehat\rho_i)$
is the relative potential energy density, defined by
$$
  F_i(\rho_i,\na\rho_i|\widehat\rho_i,\na\widehat\rho_i)
	= F_i - \widehat F_i - \frac{\pa\widehat F_i}{\pa\rho_i}(\rho_i-\widehat\rho_i)
	- \frac{\pa \widehat F_i}{\pa \na \rho_i}\cdot\na(\rho_i-\widehat\rho_i),
$$
with $F_i=F_i(\rho_i,\na\rho_i)$ and 
$\widehat F_i=F_i(\widehat\rho_i,\na\widehat\rho_i)$. This functional
satisfies a relative entropy inequality, proved in Proposition \ref{prop.rei}
for solutions to \eqref{1.eq1}-\eqref{1.eq2} and \eqref{1.CE1}-\eqref{1.CE2}
and in Proposition \ref{prop.rei2} for solutions to \eqref{1.eq1}-\eqref{1.eq2} 
and \eqref{1.lim}. The relative entropy approach has the advantage of being
very elementary and to be able to treat weak solutions to the original system
\cite{GLT17,LaTz17}.

For the proof of the high-friction limit in the isentropic case, we apply
the general relaxation result in \cite{Tza05} which is also based on the
relative entropy approach. We show that the framework is sufficiently general
to include multicomponent Euler flows with friction.

The paper is organized as follows. The formal Chapman-Enskog expansion
as well as the proof of parabolicity of the first-order correction system
are performed in section \ref{sec.formal}. Section \ref{sec.rigorCE} is
devoted to the rigorous proof of the Chapman-Enskog expansion in the
Euler-Korteweg case. The high-friction limit in both the isentropic and 
Euler-Korteweg case is shown in section \ref{sec.relax}.


\section{Formal asymptotics}\label{sec.formal}

In this section we perform a Chapman-Enskog asymptotic analysis to system 
\eqref{1.eq1}-\eqref{1.eq2} as $\eps \to 0$.
As a preparation, we analyze the solvability properties of the linear system 
\begin{equation}\label{1.linnh}
  \sum_{j=1}^n b_{ij}\rho_i\rho_j(v_i-v_j) = d_i , \quad i=1,\ldots,n \, ,
\end{equation}
and the associated homogeneous system
\begin{equation}\label{1.lin}
  \sum_{j=1}^n b_{ij}\rho_i\rho_j(v_i-v_j) = 0, \quad i=1,\ldots,n.
\end{equation}

The key hypothesis for \eqref{1.lin}, to be assumed
in the whole manuscript, reads as
\begin{enumerate}[label=\bf (N)]
\item \label{N} Let $(b_{ij})\in\R^{n\times n}$ be a symmetric matrix with 
nonnegative coefficients, $b_{i j} \ge 0$. For any $\rho_1,\ldots,\rho_n>0$, system
\eqref{1.lin} has the one-dimensional null space $\operatorname{span}\{\mathbf{1}\}$,
where $\mathbf{1}=(1,\ldots,1)\in\R^n$.
\end{enumerate}

By setting $B_{i j} = b_{i j} \rho_i \rho_j$, we rewrite \eqref{1.lin} in the form
\begin{equation}\label{1p.lin}
  \sum_{j=1}^n B_{ij} (v_i - v_j  ) = 0, \quad i=1,\ldots,n.
\end{equation}
If the coefficients $B_{i j}$ are symmetric and strictly positive, $B_{ij} > 0$ for  $i \ne j$,
then hypothesis \ref{N} is automatically satisfied. Indeed, due to
the symmetry of $(B_{ij})$,
\begin{align*}
   \sum_{i,j=1}^n B_{ij}(v_i-v_j)\cdot v_i 
	&= \frac12\sum_{i,j=1}^n B_{ij}(v_i-v_j)\cdot v_i
	+ \frac12\sum_{i,j=1}^n B_{ji}(v_j-v_i)\cdot v_j \\
	&= \frac12\sum_{i,j=1}^n B_{ij}|v_i-v_j|^2.
\end{align*}
If \eqref{1p.lin} is satisfied, it follows that $v_i=v_j$ for all $i\neq j$, 
and the null space of system \eqref{1.lin} is the linear span of the vector $\mathbf{1}$.
This conclusion cannot be guaranteed if some $b_{ij}$ vanish, which makes 
necessary assumption \ref{N}. The assumption guarantees that there are no 
extraneous conservation laws associated to the frictional coefficients $b_{ij}$, 
beyond the conservation of mass and total momentum.


\subsection{Solution of a linear system}\label{sec.lin}

In the sequel,  we will need to solve the linear system
\begin{equation}\label{2.lin}
  -\sum_{j=1}^n b_{ij}\rho_i\rho_j(u_i-u_j) = d_i\quad\mbox{for }i=1,\ldots,n,
	\quad\mbox{subject to }\sum_{i=1}^n \rho_iu_i=0.
\end{equation}
We give a semi-explicit solution to such systems, 
recalling the notation $B_{ij}=b_{ij}\rho_i\rho_j$.

\begin{lemma}\label{lem.linsys}
Let $d_1,\ldots,d_n\in\R^3$ satisfy $\sum_{i=1}^n d_i=0$, $\rho_1,\ldots,\rho_n>0$,
and $(B_{ij})\in\R^{n\times n}$ be a symmetric matrix satisfying $B_{ij}\ge 0$ for all
$i,j=1,\ldots,n$. We suppose that all solutions to the homogeneous system
\begin{equation}\label{2.hom}
  \sum_{j=1}^n B_{ij}(u_i-u_j)=0, \quad i=1,\ldots,n, 
\end{equation}
lie in the space $\operatorname{span}\{\mathbf{1}\}$. Then the system
\begin{equation}\label{2.linB}
  -\sum_{j=1}^n B_{ij}(u_i-u_j) = d_i\quad\mbox{for }i=1,\ldots,n,
	\quad\mbox{subject to }\sum_{i=1}^n \rho_iu_i=0,
\end{equation}
has the unique solution 
\begin{equation}\label{2.sol}
  \rho_iu_i = -\sum_{j,k=1}^{n-1}\bigg(\delta_{ij}\rho_i-\frac{\rho_i\rho_j}{\rho}
	\bigg)\tau^{-1}_{jk}d_k, \quad 
	\rho_n u_n =-\sum_{j=1}^{n-1}\rho_ju_j,
\end{equation}
where $i=1,\ldots,n$, $\rho=\sum_{i=1}^n \rho_i>0$ and 
$(\tau_{ij}^{-1})\in\R^{(n-1)\times(n-1)}$ 
is the inverse of a regular submatrix, obtained from reordering the matrix
$(\tau_{ij})\in\R^{n\times n}$ of rank $n-1$ with coefficients
$$
  \tau_{ij} = \delta_{ij}\sum_{k=1}^n B_{ik} - B_{ij}, \quad i,j=1,\ldots,n.
$$
\end{lemma}

\begin{proof}
We proceed similarly as in \cite[Section 4]{YYZ15}. The idea is to formulate
the linear system in $n-1$ equations and to invert the resulting linear system
semi-explicitly. First, we notice that we can write \eqref{2.hom} as
$$
\sum_{j=1}^n \tau_{ij}u_j=0  \quad \mbox{ for }i=1,\ldots,n,
$$
where
\begin{equation}\label{eq:taudef}
 	\tau_{ij}=-b_{ij}\rho_i\rho_j\mbox{ for } i \ne j\quad\text{and}\quad
 \tau_{ii}= -\sum_{j=1,\,j\neq i}^n \tau_{ij}.
\end{equation}  
Since we assumed that all solutions to this system lie in the space
$\operatorname{span}\{\mathbf{1}\}$, the matrix $(\tau_{ij})\in\R^{n\times n}$
has rank $n-1$. Thus, there exists an invertible submatrix 
$\tau=(\tau_{ij})\in\R^{(n-1)\times(n-1)}$ (possibly after reordering of the indices).

The linear system \eqref{2.linB} can be formulated in terms of the first $n-1$
variables. Indeed, since $\sum_{j=1}^n\tau_{ij}=0$, we find that
$$
  -d_i = \sum_{j=1}^n\tau_{ij} u_j 
	= \sum_{j=1}^{n-1}\tau_{ij} u_j + \tau_{in}u_n
	= \sum_{j=1}^{n-1}\tau_{ij}u_j - \sum_{j=1}^{n-1}\tau_{ij}u_n
	= \sum_{j=1}^{n-1}\tau_{ij}(u_j-u_n).
$$
Using the property $\rho_nu_n=-\sum_{k=1}^{n-1}\rho_ku_k$, it follows that
\begin{align}
  -d_i = \sum_{j=1}^{n-1}\tau_{ij}\bigg(u_j + \frac{1}{\rho_n}\sum_{k=1}^{n-1}
	\rho_ku_k\bigg)
	= \sum_{j,k=1}^{n-1}\tau_{ij}\bigg(\frac{1}{\rho_j}\delta_{jk} + \frac{1}{\rho_n}
	\bigg)\rho_ku_k 
	= \sum_{j,k=1}^{n-1}\tau_{ij}Q_{jk}\rho_ku_k,  \label{eq:taucal}
\end{align}  
where $Q_{ij}=\delta_{ij}\rho_j^{-1}+\rho_n^{-1}$ for $i,j=1,\ldots,n-1$.

The matrix $Q=(Q_{ij})\in\R^{(n-1)\times(n-1)}$
is invertible with inverse $(Q_{ij}^{-1})$, where
$Q_{ij}^{-1}=\delta_{ij}\rho_j-\rho_i\rho_j/\rho$ for $i,j=1,\ldots,n-1$.
Indeed, a straightforward computation shows that
\begin{align*}
  \sum_{k=1}^{n-1}Q_{ik}Q^{-1}_{kj}
	&= \sum_{k=1}^{n-1}\bigg(\frac{1}{\rho_k}\delta_{ik} + \frac{1}{\rho_n}\bigg)
	\bigg(\delta_{kj}\rho_j-\frac{\rho_k\rho_j}{\rho}\bigg) \\
  &= \delta_{ij} + \frac{\rho_j}{\rho_n} - \frac{\rho_j}{\rho}
	-  \frac{\rho_j}{\rho_n\rho}  \sum_{k=1}^{n-1}\rho_k
	= \delta_{ij}, \\
	\sum_{k=1}^{n-1}Q^{-1}_{ik}Q_{kj}
	&= \sum_{k=1}^{n-1}\bigg(\delta_{ik}\rho_k-\frac{\rho_i\rho_k}{\rho}\bigg)
	\bigg(\frac{1}{\rho_k}\delta_{kj} + \frac{1}{\rho_n}\bigg) \\
	&= \delta_{ij} - \frac{\rho_i}{\rho} + \frac{\rho_i}{\rho_n}
	-  \frac{\rho_i}{\rho\rho_n}  \sum_{k=1}^{n-1}\rho_k
	= \delta_{ij}.
\end{align*}

Thus, the matrix product $\tau Q$ is invertible with inverse $Q^{-1}\tau^{-1}$,
and we infer that
$$
  \rho_i u_i = -\sum_{j,k=1}^{n-1}Q_{ij}^{-1}\tau_{jk}^{-1}d_k
	= -\sum_{j,k=1}^{n-1}\bigg(\delta_{ij}\rho_i-\frac{\rho_i\rho_j}{\rho}\bigg)
  \tau^{-1}_{jk}d_k, \quad i=1,\ldots,n-1.
$$
This ends the proof.
\end{proof}


\subsection{Formal derivation of the Chapman-Enskog expansion}\label{sec.CE}

We perform a formal Chapman-Enskog expansion of \eqref{1.eq1}-\eqref{1.eq2}
in the high-friction regime, i.e.\ for small $\eps>0$.
We introduce the moments
$$
  \rho = \sum_{i=1}^n\rho_i, \quad \rho v = \sum_{i=1}^n\rho_iv_i,
$$
and the relative velocities $u_i=v_i-v$ for $i=1,\ldots,n$. This corresponds
to a change of variables $(v_1,\ldots,v_n)\mapsto(v,u_1,\ldots,u_n)$.
Then system \eqref{1.eq1}-\eqref{1.eq2} becomes
\begin{align}
  \pa_t\rho_i + \diver(\rho_iu_i + \rho_i v) &= 0, \label{2.u1} \\
  \pa_t(\rho_iu_i+\rho_iv) + \diver\big(\rho_i(u_i+v)\otimes(u_i+v)\big)
	&= -\rho_i\na\frac{\delta\E}{\delta\rho_i}(\brho) - \frac{1}{\eps}\sum_{j=1}^n
	b_{ij}\rho_i\rho_j(u_i-u_j), \label{2.u2}
\end{align}
subject to the constraint
\begin{equation}\label{2.constr}
  \sum_{i=1}^n \rho_iu_i = \sum_{i=1}^n\rho_i(v_i-v) 
	= \sum_{i=1}^n\rho_i v_i - \rho v = 0.
\end{equation}

The objective is to derive an effective equation in the spirit of the
Chapman-Enskog expansion for the high-friction dynamics of system
\eqref{2.u1}-\eqref{2.u2} subject to \eqref{2.constr}. For this, we introduce
the Hilbert expansion
\begin{align}
  \rho_i &= \rho_i^0 + \eps\rho_i^1+\eps^2\rho_i^2 + O(\eps^3), \nonumber \\
	u_i &= u_i^0 + \eps u_i^1 + \eps^2 u_i^2 + O(\eps^3), \label{2.hilbert} \\
	v &= v^0 + \eps v^1 + O(\eps^2). \nonumber
\end{align}
Inserting this expansion into $\rho=\sum_{i=1}^n\rho_i$, we find that
\begin{equation}\label{2.rho}
  \rho = \rho^0 + \eps\rho^1 + O(\eps^2), \quad\mbox{where }
	\rho^0 := \sum_{i=1}^n\rho_i^0,\ \rho^1 := \sum_{i=1}^n\rho_i^1,
\end{equation}
and the constraint \eqref{2.constr} leads to
$$
  0 = \sum_{i=1}^n\rho_iu_i 
	= \sum_{i=1}^n\rho_i^0 u_i^0 + \eps\sum_{i=1}^n\big(\rho_i^0u_i^1+\rho_i^1u_i^0\big)
	+ O(\eps^2).
$$
Equating terms of the same order gives
\begin{equation}\label{2.zero}
  \sum_{i=1}^n\rho_i^0 u_i^0=0, \quad
	\sum_{i=1}^n\big(\rho_i^0u_i^1+\rho_i^1u_i^0\big)=0.
\end{equation}

Next, we insert the Hilbert expansion \eqref{2.hilbert} into system 
\eqref{2.u1}-\eqref{2.u2} and identify terms of the same order: \\
$\bullet$ Terms of order $O(1/\eps)$:
\begin{equation}\label{2.H-1}
  \sum_{j=1}^n b_{ij}\rho_i^0\rho_j^0(u_i^0-u_j^0) = 0, \quad i=1,\ldots,n.
\end{equation}
$\bullet$ Terms of order $O(1)$:
\begin{align}
   &\pa_t\rho_i^0  + \diver(\rho_i^0 u_i^0+\rho_i^0v^0) = 0, 
	\label{2.H01} \\
   &\pa_t(\rho_i^0u_i^0+\rho_i^0v^0)
	+ \diver\big(\rho_i^0(u_i^0+v^0)\otimes(u_i^0+v^0)\big) \label{2.H02} \\
	&\qquad= -\rho_i^0\na\frac{\delta\E}{\delta\rho_i}(\brho^0)
	- \sum_{j=1}^nb_{ij}\rho_i^0\rho_j^0(u_i^1-u_j^1)
	- \sum_{j=1}^nb_{ij}(\rho_i^1\rho_j^0+\rho_i^0\rho_j^1)(u_i^0-u_j^0).
	\nonumber
\end{align}
$\bullet$ Terms of order $O(\eps)$: 
\begin{align}
  & \pa_t\rho_i^1 + \diver\big(\rho_i^1(u_i^0+v^0) + \rho_i^0(u_i^1+v^1)\big) = 0,
	 \label{2.H11} \\
	& \pa_t\big(\rho_i^1(u_i^0+v^0) + \rho_i^0(u_i^1+v^1)\big) 
	+ \diver\big(\rho_i^1(u_i^0+v^0)\otimes(u_i^0+v^0) \label{2.H12} \\
	&\phantom{xxxx}{}
	+ \rho_i^0(u_i^1+v^1)\otimes(u_i^0+v^0) + \rho_i^0(u_i^0+v^0)\otimes(u_i^1+v^1)\big)
	\nonumber \\
	&\phantom{xx}{}= -\rho_i^1\na\frac{\delta\E}{\delta\rho_i}(\brho^0)
	- \rho_i^0\na\bigg(\sum_{j=1}^n\frac{\delta^2\E}{\delta\rho_i\delta\rho_j}
	(\brho^0)\rho_j^1\bigg) \nonumber \\
	&\phantom{xxxx}{}- \sum_{j=1}^n b_{ij}\Big(\rho_i^0\rho_j^0(u_i^2-u_j^2)
	+ (\rho_i^1\rho_j^0+\rho_i^0\rho_j^1)(u_i^1-u_j^1) \nonumber \\
	&\phantom{xxxx}{}
	+ (\rho_i^0\rho_j^2+\rho_i^1\rho_j^1+\rho_i^2\rho_j^0)(u_i^0-u_j^0)\Big). \nonumber 
\end{align}

First, we consider equations \eqref{2.H-1} of order $O(1/\eps)$. By
assumption \ref{N} on page \pageref{N}, the first constraint in \eqref{2.zero}, 
and Lemma \ref{lem.linsys}, we deduce that $u_i^0=0$ for
$i=1,\ldots,n$, which simplifies equations \eqref{2.H01}-\eqref{2.H12}. 
Then, summing \eqref{2.H02} from $i=1,\ldots,n$ and using the symmetry of
$(b_{ij})$, $(\rho_1^0,\ldots,\rho_n^0,v^0)$ 
can be determined by solving the closed system
\begin{align}
  \pa_t\rho_i^0 + \diver(\rho_i^0 v^0) &= 0, \label{2.I01} \\
	\pa_t \bigg(\sum_{i=1}^n\rho_i^0v^0\bigg) 
	+ \diver\bigg(\sum_{i=1}^n\rho_i^0v^0\otimes v^0\bigg) 
	&= -\sum_{i=1}^n\rho_i^0\na\frac{\delta\E}{\delta\rho_i}(\brho^0). \label{2.I02}
\end{align}
It follows from \eqref{2.H02} that $u_1^1,\ldots,u_n^1$ satisfy the linear system
\begin{align}
  & -\sum_{j=1}^n b_{ij}\rho_i^0\rho_j^0(u_i^1-u_j^1) = d_i^0, \label{2.linsys2} \\
  & \mbox{where}\quad d_i^0 = \pa_t(\rho_i^0v^0) + \diver(\rho_i^0v^0\otimes v^0)
	+ \rho_i^0\na\frac{\delta\E}{\delta\rho_i}(\brho^0). \nonumber
\end{align}
Since $u_i^0=0$, the second constraint in \eqref{2.zero} 
becomes $\sum_{i=1}^n\rho_i^0 u_i^1=0$.
Moreover, \eqref{2.I02} is equivalent to $\sum_{i=1}^n d_i^0=0$, which ensures
the solvability of \eqref{2.linsys2}. By Lemma \ref{lem.linsys}, there exists a
unique solution $(u_1^1,\ldots,u_n^1)$ to \eqref{2.linsys2}. 

Next, we focus on the terms \eqref{2.H11}-\eqref{2.H12} of order $O(\eps)$.
We rewrite these equations using $u_i^0=0$ and the constraint 
$\sum_{i=1}^n\rho_i^0 u_i^1=0$ as
\begin{align}
  & \pa_t\rho_i^1 + \diver(\rho_i^1v^0 + \rho_i^0v^1) = -\diver(\rho_i^0u_i^1), 
	\label{2.I11} \\
	& \pa_t\bigg(\sum_{i=1}^n\rho_i^1v^0 + \sum_{i=1}^n\rho_i^0v^1\bigg)
	+ \diver\bigg(\sum_{i=1}^n\rho_i^1 v^0\otimes v^0
	+ \sum_{i=1}^n\rho_i^0(v^1\otimes v^0+v^0\otimes v^1)\bigg) \label{2.I12} \\
	&\phantom{xx}{}= -\sum_{i=1}^n\bigg\{\rho_i^1\na\frac{\delta\E}{\delta\rho_i}(\brho^0)
	+ \rho_i^0\na\bigg(\sum_{j=1}^n\frac{\delta^2\E}{\delta\rho_i\delta\rho_j}
	(\brho^0)\rho_j^1\bigg)\bigg\}. \nonumber
\end{align}
This is a closed system providing $(\rho_1^1,\ldots,\rho_n^1,v^1)$. 

The last task is to reconstruct the effective equations that are valid
asymptotically up to order $O(\eps^2)$. We are adding \eqref{2.I01} and
$\eps$ times \eqref{2.I11} as well as \eqref{2.I02} and 
$\eps$ times \eqref{2.I12}:
\begin{align*}
  &\pa_t(\rho_i^0+\eps\rho_i^1) + \diver\big(\rho_i^0v^0
	+ \eps(\rho_i^1v^0+\rho_i^0v^1)\big) = -\eps\diver(\rho_i^0 u_i^1), \\
  & \pa_t\big(\rho^0v^0 + \eps(\rho^1v^0+\rho^0v^1)\big)
	+ \diver\big({\rho}^0v^0\otimes v^0 + \eps(\rho^1v^0\otimes v^0
	+ \rho^0 v^1\otimes v^0 + \rho^0v^0\otimes v^1)\big) \\
	&\phantom{xx}{}= -\sum_{i=1}^n\rho_i^0\na\frac{\delta\E}{\delta\rho_i}(\brho^0)
	- \eps\sum_{i=1}^n\bigg\{\rho_i^1\na\frac{\delta\E}{\delta\rho_i}(\brho^0)
	+ \rho_i^0\na\bigg(\sum_{j=1}^n\frac{\delta^2\E}{\delta\rho_i\delta\rho_j}
	(\brho^0)\rho_j^1\bigg)\bigg\},
\end{align*}
where $\rho^0$ and $\rho^1$ are defined in \eqref{2.rho}. With the notation
\begin{align*}
  & \rho_i^\eps = \rho_i^0 + \eps\rho_i^1 + O(\eps^2), \quad 
	u_i^\eps = u_i^0 + \eps u_i^1 +O(\eps^2), \\
	& v^\eps = v^0 + \eps v^1 + O(\eps^2), \quad 
	\rho^\eps = \sum_{i=1}^n\rho_i^\eps,
\end{align*}
and recalling that $u_i^0=0$, we infer that 
$(\rho_1^\eps,\ldots,\rho_n^\eps,v^\eps)$ satisfies
\begin{align*}
  \pa_t\rho_i^\eps + \diver(\rho_i^\eps v^\eps) 
	&= -\diver(\rho_i^\eps u_i^\eps) + O (\eps^2),
	\\	
  \pa_t(\rho^\eps v^\eps) + \diver(\rho^\eps v^\eps\otimes v^\eps)
	&= -\sum_{i=1}^n\rho_i^\eps\na\frac{\delta\E}{\delta\rho_i}(\brho^\eps) + O(\eps^2).
\end{align*}

It remains to reconstruct the formula determining $(u_1^\eps,\ldots,u_n^\eps)$.
We deduce from \eqref{2.linsys2} that
\begin{equation}\label{2.d}
  -\sum_{j=1}^n b_{ij}\rho_i^\eps\rho_j^\eps(u_i^\eps-u_j^\eps)
	= -\eps\sum_{j=1}^n b_{ij}\rho_i^0\rho_j^0(u_i^1-u_j^1) + O(\eps^2)
	= \eps d_i^0 + O(\eps^2).
\end{equation}
The variables $d_i^0$ can be expressed in terms of $\brho^0$ only. Indeed,
since $\pa_t\rho_i^0+\diver(\rho_i^0v^0)=0$ and $\pa_t\rho^0+\diver(\rho^0v^0)=0$,
it follows that
\begin{align*}
  d_i^0 &= \big(\pa_t\rho_i^0+\diver(\rho_i^0v^0)\big)v^0
	+ \rho_i^0\big(\pa_t v^0 + v^0\cdot\na v^0\big)
	+ \rho_i^0\na\frac{\delta\E}{\delta\rho_i}(\brho^0) \\
	&= \rho_i^0\big(\pa_t v^0 + v^0\cdot\na v^0\big)
	+ \rho_i^0\na\frac{\delta\E}{\delta\rho_i}(\brho^0) \\
	&= \frac{\rho_i^0}{\rho^0}\big(\pa_t(\rho^0v^0) + \diver(\rho^0v^0\otimes v^0)\big)
	+ \rho_i^0\na\frac{\delta\E}{\delta\rho_i}(\brho^0) \\
  &= -\frac{\rho_i^0}{\rho^0}\sum_{j=1}^n\rho_j^0\na\frac{\delta\E}{\delta\rho_j}
	(\brho^0) + \rho_i^0\na\frac{\delta\E}{\delta\rho_i}(\brho^0),
\end{align*}
where in the last step we have used \eqref{2.I02}. This motivates us to define
\begin{equation}\label{2.deps}
  d_i^\eps := -\frac{\rho_i^\eps}{\rho^\eps}\sum_{j=1}^n\rho_j^\eps
	\na\frac{\delta\E}{\delta\rho_j}(\brho^\eps) 
	+ \rho_i^\eps\na\frac{\delta\E}{\delta\rho_i}(\brho^\eps).
\end{equation}
Hence, we can formulate \eqref{2.d} as
$$
  -\sum_{j=1}^n b_{ij}\rho_i^\eps\rho_j^\eps(u_i^\eps-u_j^\eps) 
	= \eps d_i^\eps + O(\eps^2).
$$
The constraints $\sum_{i=1}^n\rho_i^0 u_i^0 = 0$ and $\sum_{j=1}^n\rho_i^0u_i^1=0$ from \eqref{2.zero} imply that
$$
  \sum_{i=1}^n\rho_i^\eps u_i^\eps = \sum_{i=1}^n\rho_i^0 u_i^0+  \eps\sum_{i=1}^n \rho_i^0u_i^1 + O(\eps^2)
	= O(\eps^2).
$$
As the functions $\rho_i^\eps$, $v^\eps$, and $u_i^\eps$ are defined only up to 
order $O(\eps^2)$, we may set $\sum_{i=1}^n \rho_i^\eps u_i^\eps=0$ up to that
order.

We summarize our calculations. The functions 
$(\brho^\eps,v^\eps)=(\rho_1^\eps,\ldots,\rho_n^\eps,v^\eps)$
satisfy up to order $O(\eps^2)$ the effective equations
\begin{align}
  \pa_t\rho_i^\eps + \diver(\rho_i^\eps v^\eps) 
	&= -\diver(\rho_i^\eps u_i^\eps), \label{2.eff1} \\	
  \pa_t(\rho^\eps v^\eps) + \diver(\rho^\eps v^\eps\otimes v^\eps)
	&= -\sum_{i=1}^n\rho_i^\eps\na\frac{\delta\E}{\delta\rho_i}(\brho^\eps),
	\label{2.eff2}
\end{align}
where  $\rho^\eps=\sum_{i=1}^n\rho^\eps_i$, and $\bm{u}^\eps=(u_1^\eps,\ldots,u_n^\eps)$ is
the unique solution to
\begin{equation}\label{2.effu}
  -\sum_{j=1}^n b_{ij}\rho_i^\eps\rho_j^\eps(u_i^\eps-u_j^\eps) = \eps d_i^\eps,
	\quad \sum_{j=1}^n\rho_j^\eps u_j^\eps=0,
\end{equation}
for $i=1,\ldots,n$, where $d_i^\eps$ is defined in \eqref{2.deps}.


\subsection{Gradient-flow structure and parabolicity}\label{sec.para}

We show that the effective equations have a formal gradient-flow structure
and, if the total mass is constant, a parabolic structure 
in the sense of Petrovskii \cite{Ama93}. 
First, we reformulate system \eqref{2.eff1}-\eqref{2.effu}.

\begin{lemma}[Gradient-flow structure]\label{lem.diff}
System \eqref{2.eff1}-\eqref{2.effu} can be rewritten as
\begin{align*}
  \pa_t\rho_i^\eps + \diver(\rho_i^\eps v^\eps) 
	&= \eps\diver\sum_{j=1}^n D_{ij}^\eps\na\frac{\delta\E}{\delta\rho_j}(\brho^\eps), \\
  \pa_t(\rho^\eps v^\eps) + \diver(\rho^\eps v^\eps\otimes v^\eps)
	&= -\sum_{i=1}^n \rho_i^\eps\na\frac{\delta\E}{\delta\rho_i}(\brho^\eps),
\end{align*}
where $i=1,\ldots,n$, $\rho^\eps=\sum_{i=1}^n\rho_i^\eps$, and 
$$
  D^\eps = G(Q^\eps)^{-1}(\tau^\eps)^{-1}(Q^\eps)^{-1} G^\top\in\R^{n\times n},
$$ 
where $(Q^\eps)^{-1}\in\R^{(n-1)\times(n-1)}$ has the coefficients
$(Q^\eps)_{ij}^{-1}=\delta_{ij}\rho_i^\eps-\rho_i^\eps\rho_j^\eps/\rho^\eps$,
$(\tau^\eps)^{-1}$ is the inverse of the $(n-1)\times(n-1)$ matrix introduced
in Lemma \ref{lem.linsys}, and $G=(G_{ij})\in\R^{n\times(n-1)}$ is defined
by $G_{ii}=1$, $G_{ni}=-1$ for $i=1,\ldots,n-1$, and $G_{ij}=0$ elsewhere. 
\end{lemma}

\begin{proof}
In view of Lemma \ref{lem.linsys}, the solution to \eqref{2.effu} can be expressed as
\begin{equation}\label{2.aux1}
  \rho_i^\eps u_i^\eps = -\eps\sum_{j,k=1}^{n-1}\bigg(\delta_{ij}\rho_i^\eps
	- \frac{\rho_i^\eps\rho_j^\eps}{\rho^\eps}\bigg)({\tau}^\eps)_{jk}^{-1} d_k^\eps,
	\quad i=1,\ldots,n-1,
\end{equation}
where $({\tau}^\eps)^{-1}=(({\tau}^\eps)_{jk}^{-1})$ 
is the inverse of a regular matrix in $\R^{(n-1)\times(n-1)}$ whose coefficients
only depend on $b_{ij}\rho_i^\eps\rho_j^\eps$. We wish to reformulate $d_i^\eps$
in terms of $\brho^\eps$. For this, we compute, using
$\rho^\eps_n = \rho^\eps - \sum_{j=1}^{n-1}\rho_j^\eps$,
\begin{align*}
  d_i^\eps &= \rho_i^\eps\na\frac{\delta\E}{\delta\rho_i}(\brho^\eps)
	- \frac{\rho_i^\eps}{\rho^\eps}\sum_{j=1}^n
	\rho_j^\eps\na\frac{\delta\E}{\delta\rho_j}(\brho^\eps) \\
	&= \sum_{j=1}^{n-1}\bigg(\delta_{ij}\rho_i^\eps 
	- \frac{\rho_i^\eps\rho_j^\eps}{\rho^\eps}\bigg)
	\na\frac{\delta\E}{\delta\rho_j}(\brho^\eps)
	- \frac{\rho_i^\eps\rho_n^\eps}{\rho^\eps}\na\frac{\delta\E}{\delta\rho_n}
	(\brho^\eps) \\
	&= \sum_{j=1}^{n-1}(Q^\eps)_{ij}^{-1}\na\frac{\delta\E}{\delta\rho_j}(\brho^\eps)
	- \frac{\rho_i^\eps}{\rho^\eps}\bigg(\rho^\eps 
	- \sum_{j=1}^{n-1}\rho_j^\eps\bigg)\na\frac{\delta\E}{\delta\rho_n}(\brho^\eps) \\
	&= \sum_{j=1}^{n-1}(Q^\eps)_{ij}^{-1}\na\frac{\delta\E}{\delta\rho_j}(\brho^\eps)
	- \sum_{j=1}^{n-1}\bigg(\delta_{ij}\rho_i^\eps 
	- \frac{\rho_i^\eps\rho_j^\eps}{\rho^\eps}\bigg)
	\na\frac{\delta\E}{\delta\rho_n}(\brho^\eps) \\
	&= \sum_{j=1}^{n-1}(Q^\eps)_{ij}^{-1}\na\bigg(\frac{\delta\E}{\delta\rho_j}(\brho^\eps)
	- \frac{\delta\E}{\delta\rho_n}(\brho^\eps)\bigg).
\end{align*}
Inserting this expression into \eqref{2.aux1} gives
\begin{align}
  \rho_i^\eps u_i^\eps 
	&= -\eps\sum_{j,k,\ell=1}^{n-1}(Q^\eps)_{ij}^{-1}
	(\tau^\eps)_{jk}^{-1}(Q^\eps)_{k\ell}^{-1}
	\na\bigg(\frac{\delta\E}{\delta\rho_\ell}(\brho^\eps)
	- \frac{\delta\E}{\delta\rho_n}(\brho^\eps)\bigg) \nonumber \\
  &= -\eps\sum_{\ell=1}^{n-1} \widetilde D^\eps_{i\ell}
	\na\bigg(\frac{\delta\E}{\delta\rho_\ell}(\brho^\eps)
	- \frac{\delta\E}{\delta\rho_n}(\brho^\eps)\bigg), \label{2.rhou}
\end{align}
with $\widetilde D_{ij}^\eps$  the elements of the invertible matrix
$\widetilde D^\eps=(Q^\eps)^{-1}(\tau^\eps)^{-1}(Q^\eps)^{-1}\in\R^{(n-1)\times(n-1)}$.
Finally, setting $D^\eps:=G\widetilde D^\eps G^\top$,
we can formulate \eqref{2.rhou} as
\begin{align}\label{eq:rhou}
  \rho_i^\eps u_i^\eps = -\eps\sum_{j=1}^n D_{ij}^\eps\na\frac{\delta\E}{\delta\rho_j}
	(\brho^\eps),	\quad i = 1,\ldots, n.
\end{align}
Note that in this writing, the last row of the matrix expresses the constraint
$\rho_n u_n = - \sum_{j=1}^{n-1} \rho_j u_j$.
We finish the proof after inserting this expression into \eqref{2.eff1}. 
\end{proof}

Let $v^\eps=0$. Then the sum of \eqref{2.eff1} over $i=1,\ldots,n$
yields, because of $\sum_{i=1}^n\rho_i^\eps u_i^\eps=0$,
$\pa_t\rho^\eps = 0$. Thus, $\rho^\eps$ does not depend on time and is fixed
by the initial total mass. It is sufficient to consider 
$\widetilde\brho^\eps:=(\rho_1^\eps,\ldots,\rho_{n-1}^\eps)$ since
the last component can be recovered from $\rho_n^\eps
=\rho^\eps-\sum_{i=1}^{n-1}\rho_i^\eps$. Accordingly, the energy
can be formulated as a function of the variable $\widetilde\brho^\eps$:
\begin{equation}
\label{redenergy}
  \widetilde\E(\widetilde\brho^\eps) := \E\bigg(\rho_1^\eps,\ldots,\rho_{n-1}^\eps,
	\rho^\eps-\sum_{i=1}^{n-1}\rho_i^\eps\bigg).
\end{equation}

\begin{lemma}[Parabolicity in the sense of Petrovskii]\label{lem.Estar}
Let $(\brho^\eps,v^\eps)$ be a solution to
\eqref{2.eff1}-\eqref{2.eff2} with $v^\eps =0$ and let 
$\bm{u}^\eps$ be a solution to \eqref{2.effu}. 
Suppose that $\E(\brho^\eps)$ is strictly convex. Then $\brho^\eps$ solves
\begin{equation}\label{2.Estar}
  \pa_t\rho_i^\eps = \eps\diver\sum_{j=1}^{n-1}\widetilde D_{ij}^\eps
	\na\frac{\delta\widetilde\E}{\delta\rho_j}
	(\widetilde\brho^\eps), \quad i=1,\ldots,n-1,
\end{equation}
the matrix $\widetilde D^\eps=(\widetilde D_{ij}^\eps)$ is positive definite,
and the energy $\widetilde\E$ is a Lyapunov functional 
along solutions to \eqref{2.Estar}:
$$
  \frac{d\widetilde\E}{dt} = -\eps\int_{\R^3}\sum_{i,j=1}^n \widetilde D_{ij}^\eps
	\na\frac{\delta\widetilde\E}{\delta\rho_i}\cdot
	\na\frac{\delta\widetilde\E}{\delta\rho_j}dx \le 0.
$$
Moreover, if $\rho_i^\eps>0$ for $i=1,\ldots,n$,
all eigenvalues of $\widetilde D^\eps\widetilde\E''$ are real and positive
(here, $\widetilde\E''=d^2\widetilde\E/d\widetilde{\brho}^2$ 
is the Hessian of the energy $\widetilde\E$). This means that
\eqref{2.Estar} is parabolic in the sense of Petrovskii.
\end{lemma}

A second-order system is called {\it parabolic in the sense of Petrovskii}
if the real parts of the eigenvalues of the diffusion matrix are positive;
see \cite[Remark 4.2a]{Ama93}.

\begin{proof}
Since the variational derivative of $\widetilde\E$ equals
$$
  \frac{\delta\widetilde\E}{\delta\rho_i}(\widetilde{\brho}^\eps)
	= \frac{\delta\E}{\delta\rho_i}(\brho^\eps)
	- \frac{\delta\E}{\delta\rho_n}(\brho^\eps), \quad i=1,\ldots,n-1,
$$
expression \eqref{2.rhou} in the proof of Lemma \ref{lem.diff} shows that 
for $i=1,\ldots,n-1$,
$$
  \eps\sum_{j=1}^n D_{ij}^\eps\na\frac{\delta\E}{\delta\rho_j}(\brho^\eps)
	= \eps\sum_{j=1}^{n-1} \widetilde D^\eps_{ij}
	\na\bigg(\frac{\delta\E}{\delta\rho_j}(\brho^\eps)
	- \frac{\delta\E}{\delta\rho_n}(\brho^\eps)\bigg)
	= \eps\sum_{j=1}^{n-1} \widetilde D^\eps_{ij}
	\na\frac{\delta\widetilde\E}{\delta\rho_j}(\widetilde{\brho}^\eps),
$$
proving \eqref{2.Estar}. Next, we show that $\widetilde D^\eps$ is positive definite.
As $(b_{ij})$ is a symmetric matrix with nonnegative entries 
(by assumption \ref{N} on page \pageref{N}), the matrix
$$
\tau^\eps_{ij} = \delta_{ij}\sum_{k=1}^n b_{ik}\rho_i^\eps\rho_k^\eps
- b_{ij}\rho_i^\eps\rho_j^\eps
$$ 
is symmetric, diagonally dominant, and has
real nonnegative diagonal elements. Therefore, $(\tau_{ij}^\eps)$ 
is positive semidefinite.
We know from the proof of Lemma \ref{lem.linsys} that there exists an invertible
$(n-1)\times(n-1)$ submatrix $({\tau}^\eps)_{ij}^{-1}$. This submatrix is symmetric,
positive semidefinite, and invertible, so all its eigenvalues must be positive and,
in fact, it is positive definite.
Moreover, since $(Q^\eps)^{-1}$ is regular, 
$\widetilde D^\eps=(Q^\eps)^{-1}(\tau^\eps)^{-1}(Q^\eps)^{-1}$ is positive definite.

It remains to show that $\widetilde D^\eps\widetilde\E''$ has only real and positive
eigenvalues. We claim that $\widetilde\E''$ is positive definite. To see this, 
we calculate (dropping the superindex $\eps$)
$$
  \widetilde\E'' = \frac{d}{d\widetilde\brho}
	\bigg(\frac{d\widetilde\E}{d\brho}
	\frac{d\brho}{d\widetilde\brho}\bigg)
	= \bigg(\frac{d\brho}{d\widetilde\brho}\bigg)^\top
	\frac{d^2\widetilde\E}{d\brho^2}
	\bigg(\frac{d\brho}{d\widetilde\brho}\bigg)
	+ \frac{d\widetilde\E}{d\brho}\frac{d^2\brho}{d\widetilde\brho^2}.
$$
Since $\brho=(\rho_1,\ldots,\rho_{n-1},\rho-\sum_{i=1}^{n-1}\rho_i)$, we have
$$
  \frac{d\brho}{d\widetilde\brho} = \begin{pmatrix}
	1      & 0 & \cdots & 0 \\
	0      & 1 &        & \vdots \\
	\vdots &   & \ddots & 0 \\
	0      & 0 &        & 1 \\
	-1     &-1 & \cdots &-1 \end{pmatrix}\in \R^{n\times (n-1)},
$$
and $d^2\brho/d\widetilde\brho^2$ vanishes since the transformation 
$\widetilde\brho\mapsto\brho$ is linear. 
By the strict convexity of $\widetilde\E$, there exists $\kappa>0$ such that 
for any $z=(z_1,\ldots,z_{n-1})\in\R^{n-1}$,
$$
  z^\top\widetilde\E''z = z^\top\bigg(\frac{d\brho}{d\widetilde\brho}\bigg)^\top
	\frac{d^2\widetilde\E}{d\brho^2}
	\bigg(\frac{d\brho}{d\widetilde\brho}\bigg)z
	\ge \kappa\bigg|\frac{d\brho}{d\widetilde\brho}z\bigg|^2
	= \kappa\sum_{i=1}^{n-1}z_i^2 + \kappa\bigg(\sum_{i=1}^{n-1}z_i\bigg)^2
  \ge \kappa|z|^2.
$$
This shows that $\widetilde\E''$ is symmetric and positive definite. Since also 
$\widetilde D^\eps$ is symmetric and positive definite, 
Proposition 6.1 of \cite{Ser10} implies that the eigenvalues of 
$\widetilde D^\eps\widetilde\E''$ are real and positive.
\end{proof}


\section{Justification of the Chapman-Enskog expansion}\label{sec.rigorCE}

In this section, we justify the validity of the Chapman-Enskog expansion 
performed in section \ref{sec.CE}. We recall that the energy is the sum
of the partial energies depending on the partial densities and their gradients,
\begin{equation}\label{3.assumptE}
  \E(\brho) = \int_{\R^3}\sum_{i=1}^n F_i(\rho_i,\na\rho_i)dx.
\end{equation}
It includes Euler-Korteweg models with the partial energy density \eqref{1.korte}.
Under this hypothesis, it is shown in \cite[formula (2.25)]{GLT17} 
that the force term in \eqref{1.eq2} can
be written as the divergence of a stress tensor $S_i$:
\begin{equation}\label{3.ES}
  -\rho_i\na\frac{\delta\E}{\delta\rho_i}(\brho) = \diver S_i(\brho), \quad
	i=1,\ldots,n,
\end{equation}
where 
\begin{align}
  & S_i(\brho) = -s_i(\rho_i,\na\rho_i) \mathbb{I} 
	+ \diver r_i(\rho_i,\na\rho_i)\mathbb{I}
	- H_i(\rho_i,\na\rho_i), \quad\mbox{and} \label{3.S} \\
	& s_i(\rho_i,q_i) = \rho_i\frac{\pa F_i}{\pa\rho_i}(\rho_i,q_i) 
	+ q_i\cdot\frac{\pa F_i}{\pa q_i}(\rho_i,q_i) - F_i(\rho_i,q_i), \nonumber \\
	& r_i(\rho_i,q_i) = \rho_i\frac{\pa F_i}{\pa q_i}(\rho_i,q_i), \nonumber \\
	& H_i(\rho_i,q_i) = q_i\otimes \frac{\pa F_i}{\pa q_i}(\rho_i,q_i), \nonumber
\end{align}
and $q_i=\nabla \rho_i$, $\mathbb{I}$ is the unit matrix in $\R^{3\times 3}$.

We consider weak solutions to the original system \eqref{1.eq1}-\eqref{1.eq2},
\begin{align}
  \pa_t\rho_i^\eps + \diver(\rho_i^\eps v_i^\eps) &= 0, \quad i=1,\ldots,n, 
	\label{3.rho} \\
	\pa_t(\rho_i^\eps v_i^\eps) + \diver(\rho_i^\eps v_i^\eps\otimes v_i^\eps)
	&= -\rho_i^\eps\na\frac{\delta\E}{\delta\rho_i}(\brho^\eps) 
	- \frac{1}{\eps}\sum_{j=1}^n b_{ij}\rho_i^\eps\rho_j^\eps
	(v_i^\eps-v_j^\eps), \label{3.rhov}
\end{align}
and strong solutions to the approximate system \eqref{2.eff1}-\eqref{2.eff2},
\begin{align}
  \pa_t\widehat\rho_i^\eps + \diver(\widehat\rho_i^\eps\widehat v^\eps)
	&= -\diver(\widehat\rho_i^\eps\widehat u_i^\eps), \quad i=1,\ldots,n, \label{3.rhoeps}\\
	\pa_t(\widehat\rho^\eps \widehat v^\eps) + \diver(\widehat\rho^\eps
	\widehat v^\eps\otimes\widehat v^\eps) 
	&= -\sum_{j=1}^n\widehat\rho_j^\eps\na\frac{\delta\E}{\delta\rho_j}
	(\widehat\brho^\eps),
	\quad \widehat\rho^\eps = \sum_{j=1}^n\widehat\rho_j^\eps, \label{3.rhoveps}
\end{align}
where $(\widehat u_1^\eps,\ldots,\widehat u_n^\eps)$ solves \eqref{2.effu},
\begin{equation}\label{3.u}
  -\sum_{j=1}^n b_{ij}\widehat\rho_i^\eps\widehat\rho_j^\eps
	(\widehat u_i^\eps-\widehat u_j^\eps) = \eps \widehat d_i^\eps, \quad
	\sum_{j=1}^n\widehat\rho_j^\eps\widehat u_j^\eps = 0,
\end{equation}
and $\widehat d_i^\eps$ is given by \eqref{2.deps},
$$
  \widehat d_i^\eps = -\frac{\widehat\rho_i^\eps}{\widehat\rho^\eps}\sum_{j=1}^n
	\widehat\rho_j^\eps\na\frac{\delta\E}{\delta\rho_j}(\widehat\brho^\eps)
	+ \widehat\rho_i^\eps\na\frac{\delta\E}{\delta\rho_i}(\widehat\brho^\eps).
$$
Our aim is to show that the difference of the solutions of
\eqref{3.rho}-\eqref{3.rhov} and \eqref{3.rhoeps}-\eqref{3.rhoveps}
converges to zero as $\eps\to 0$ in a certain sense; see Theorem \ref{thm.convCE}
below.

Lemma \ref{lem.diff} shows that system \eqref{3.rhoeps}-\eqref{3.rhoveps} 
can be written without the variable $\widehat u_i^\eps$ as a diffusion system. 
However, the current formulation is more convenient to verify the convergence result. 
In the sequel, we replace $-\rho_i\na(\delta\E/\delta\rho_i)$
by $\diver S_i$  using \eqref{3.ES}.

\subsection{Preparations}\label{sec.prep}

We reformulate the approximate system \eqref{3.rhoeps}-\eqref{3.rhoveps} 
in a form that resembles the original system \eqref{3.rho}-\eqref{3.rhov} 
with an error term:

\begin{lemma}\label{lem.reform}
Setting $\widehat v_i^\eps=\widehat v^\eps+\widehat u_i^\eps$,
system \eqref{3.rhoeps}-\eqref{3.rhoveps} is equivalent to
\begin{align}
  \pa_t\widehat\rho_i^\eps + \diver(\widehat\rho_i^\eps\widehat v_i^\eps) &= 0, 
	\label{3.rhoeps2} \\
	\pa_t(\widehat\rho_i^\eps\widehat v_i^\eps) + \diver(\widehat\rho_i^\eps
	\widehat v_i^\eps\otimes\widehat v_i^\eps)
	&= -\widehat\rho_i^\eps\na\frac{\delta\E}{\delta\rho_i}
	(\widehat\brho^\eps) 
	- \frac{1}{\eps}\sum_{j=1}^n b_{ij}\widehat\rho_i^\eps\widehat\rho_j^\eps
	(\widehat v_i^\eps-\widehat v_j^\eps) + \widehat R_i^\eps, \label{3.rhoveps2}
\end{align}
where the remainder $\widehat R_i^\eps$ is given by
\begin{equation}\label{3.Reps}
  \widehat R_i^\eps :=  -\widehat v^\eps
	\diver(\widehat\rho_i^\eps\widehat u_i^\eps)
	+ \pa_t(\widehat\rho_i^\eps\widehat u_i^\eps)
	+ \diver(\widehat\rho_i^\eps\widehat u_i^\eps\otimes\widehat v^\eps
	+ \widehat\rho_i^\eps\widehat v^\eps\otimes\widehat u_i^\eps)
	+ \diver(\widehat\rho_i^\eps\widehat u_i^\eps\otimes\widehat u_i^\eps).
\end{equation}
\end{lemma}

\begin{proof}
Equation \eqref{3.rhoeps2} follows directly from \eqref{3.rhoeps} and the definition 
$\widehat v_i^\eps=\widehat v^\eps+\widehat u_i^\eps$. 
We write the evolution of the momentum in a similar format as \eqref{3.rhov},
\begin{align*}
  \pa_t(\widehat\rho_i^\eps\widehat v_i^\eps) + \diver(\widehat\rho_i^\eps
	\widehat v_i^\eps\otimes\widehat v_i^\eps)
	= -\widehat\rho_i^\eps\na\frac{\delta\E}{\delta\rho_i}(\widehat\brho^\eps)
	- \frac{1}{\eps}\sum_{j=1}^n b_{ij}\widehat\rho_i^\eps \widehat\rho_j^\eps
	(\widehat v_i^\eps-\widehat v_j^\eps) + \widehat R_i^\eps,
\end{align*}
where $\widehat R_i^\eps$ contains the remaining terms:
\begin{align}
  \widehat R_i^\eps &= \pa_t(\widehat\rho_i^\eps\widehat v^\eps)
	+ \diver(\widehat\rho_i^\eps\widehat v^\eps\otimes\widehat v^\eps)
	+ \widehat\rho_i^\eps\na\frac{\delta\E}{\delta\rho_i}(\widehat\brho^\eps) 
	+ \frac{1}{\eps}\sum_{j=1}^n b_{ij}\widehat\rho_i^\eps
	\widehat\rho_j^\eps(\widehat v_i^\eps-\widehat v_j^\eps) \nonumber \\
	&\phantom{xx}{}+ \pa_t(\widehat\rho_i^\eps\widehat u_i^\eps)
	+ \diver(\widehat\rho_i^\eps\widehat u_i^\eps\otimes\widehat v^\eps
	+ \widehat\rho_i^\eps\widehat v^\eps\otimes\widehat u_i^\eps)
	+ \diver(\widehat\rho_i^\eps\widehat u_i^\eps\otimes\widehat u_i^\eps).
	\label{3.aux1}
\end{align}
It remains to show that this expression equals \eqref{3.Reps}.
The last three terms are already in the desired form.
By \eqref{3.u}, we have
$$
  \frac{1}{\eps}\sum_{j=1}^n b_{ij}\widehat\rho_i^\eps\widehat\rho_j^\eps
	(\widehat v_i^\eps-\widehat v_j^\eps) 
	= \frac{\widehat\rho_i^\eps}{\widehat\rho^\eps}\sum_{j=1}^n\widehat\rho_j^\eps
	\nabla \frac{\delta\E}{\delta\rho_j}(\widehat\brho^\eps) 
	- \widehat\rho_i^\eps \nabla \frac{\delta\E}{\delta\rho_i}(\widehat\brho^\eps).
$$
Therefore, we can replace the third and fourth terms in $\widehat R_i^\eps$ by
\begin{equation}\label{3.aux2}
  \frac{\widehat\rho_i^\eps}{\widehat\rho^\eps}\sum_{j=1}^n\widehat\rho_j^\eps
	\na\frac{\delta\E}{\delta\rho_j}(\widehat\brho^\eps).
\end{equation}
We reformulate the first and second terms in $\widehat R_i^\eps$.
Adding \eqref{3.rhoeps} over $i=1,\ldots,n$ and using 
$\sum_{j=1}^n\widehat\rho_j^\eps\widehat u_j^\eps=0$, we deduce that
$\pa_t\widehat\rho^\eps + \diver(\widehat\rho^\eps\widehat v^\eps)=0$.
This equation and \eqref{3.rhoeps}, \eqref{3.rhoveps} show that
\begin{align*}
  \pa_t & (\widehat\rho_i^\eps\widehat v^\eps) + \diver(\widehat\rho_i^\eps
	\widehat v^\eps\otimes\widehat v^\eps)
	= \big(\pa_t\widehat\rho_i^\eps + \diver(\widehat\rho_i^\eps \widehat v^\eps)
	\big)\widehat v^\eps + \widehat\rho_i^\eps\big(\pa_t\widehat v^\eps
	+ \widehat v^\eps\cdot\na\widehat v^\eps\big) 
\\
	&= -\diver(\widehat\rho_i^\eps \widehat u_i^\eps)\widehat v^\eps
	+ \frac{\widehat\rho_i^\eps}{\widehat\rho^\eps}
	\big(\pa_t(\widehat\rho^\eps\widehat v^\eps) - (\pa_t\widehat\rho^\eps)\widehat v^\eps
  + \widehat\rho^\eps\widehat v^\eps\cdot\na\widehat v^\eps\big) 
\\
	&= -\diver(\widehat\rho_i^\eps \widehat u_i^\eps)\widehat v^\eps
	+ \frac{\widehat\rho_i^\eps}{\widehat\rho^\eps}
	\big(\pa_t(\widehat\rho^\eps\widehat v^\eps) + \diver(\widehat\rho^\eps
	\widehat v^\eps\otimes\widehat v^\eps)\big) \\
	&= -\diver(\widehat\rho_i^\eps \widehat u_i^\eps)\widehat v^\eps
	- \frac{\widehat\rho_i^\eps}{\widehat\rho^\eps}\sum_{j=1}^n\widehat\rho_j^\eps
	\na\frac{\delta\E}{\delta\rho_j}(\widehat\brho^\eps),
\end{align*}
where we used \eqref{3.rhoveps} in the last step. The last term cancels with
\eqref{3.aux2}, showing that \eqref{3.aux1} reduces to  \eqref{3.Reps}.
\end{proof}

We need later the explicit expressions of the variational derivatives of
$\E$ and $S_i$.

\begin{lemma}[Variational derivatives of $\E$]\label{lem.second}
Let $\E$ be given by \eqref{3.assumptE}. Then,
for test functions $\psi_i$ and $\phi_i$,
\begin{align*}
  \sum_{i=1}^n\bigg\langle\frac{\delta\E}{\delta\rho_i}(\brho),\psi_i\bigg\rangle
	&= \int_{\R^3}\sum_{i=1}^n\bigg(\frac{\pa F_i}{\pa\rho_i}(\rho_i,\na\rho_i)\psi_i
	+ \frac{\pa F_i}{\pa q_i}(\rho_i,\na\rho_i)\cdot\na\psi_i\bigg)dx, \\
	\sum_{i=1}^n\bigg\langle\!\!\!\bigg\langle\frac{\delta^2\E}{\delta\rho_i^2}
	(\brho),(\psi_i,\phi_i)\bigg\rangle\!\!\!\bigg\rangle
	&= \sum_{i=1}^n\int_{\R^3}(\phi_i,\na\phi_i)
	\begin{pmatrix} \pa^2 F_i/\pa\rho_i^2 & \pa^2 F_i/\pa\rho_i\pa q_i \\
	\pa^2 F_i/\pa\rho_i\pa q_i & \pa^2 F_i/\pa q_i^2 \end{pmatrix}
	\begin{pmatrix} \psi_i \\ \na\psi_i \end{pmatrix} dx, 
\end{align*}
\end{lemma}

\begin{proof}
We compute the first variational derivative with respect to the test function
$\bm\psi=(\psi_1,\ldots,\psi_n)$:
\begin{align*}
  \sum_{i=1}^n\bigg\langle\frac{\delta\E}{\delta\rho_i}(\brho),\psi_i\bigg\rangle
	&= \frac{d}{d\tau}\E(\brho+\tau\bm\psi)\bigg|_{\tau=0}
	= \frac{d}{d\tau}\int_{\R^3}\sum_{i=1}^n F_i(\rho_i+\tau\psi_i,
	\na\rho_i+\tau\na\psi_i)dx\bigg|_{\tau=0} \\
  &= \int_{\R^3}\sum_{i=1}^n\bigg(\frac{\pa F_i}{\pa\rho_i}(\rho_i,\na\rho_i)\psi_i
	+ \frac{\pa F_i}{\pa q_i}(\rho_i,\na\rho_i)\cdot\na\psi_i\bigg)dx.
\end{align*}
Next, we calculate the second variational derivative,
where $\bm\phi=(\phi_1,\ldots,\phi_n)$:
\begin{align*}
  \sum_{i=1}^n&\bigg\langle\!\!\!\bigg\langle
	\frac{\delta^2\E}{\delta\rho_i^2}
	(\brho),(\psi_i,\phi_i)\bigg\rangle\!\!\!\bigg\rangle
	 = \frac{d}{d\tau}\bigg\langle\sum_{i=1}^n\frac{\delta\E}{\delta\rho_i}
	(\brho+\tau\bm\phi), \psi_i\bigg\rangle	\bigg|_{\tau=0}	\\
	&= \frac{d}{d\tau}\int_{\R^3}\sum_{i=1}^n\bigg(\frac{\pa F_i}{\pa\rho_i}
	\big(\rho_i+\tau\phi_i,\na(\rho_i+\tau\phi_i)\big)\psi_i \\
	&\phantom{xx}{}
	- \frac{\pa F_i}{\pa q_i}\big(\rho_i+\tau\phi_i,\na(\rho_i+\tau\phi_i)\big)
	\cdot\na\psi_i\bigg)dx\bigg|_{\tau=0} \\
	&= \sum_{i=1}^n\bigg(\frac{\pa^2 F_i}{\pa\rho_i^2}(\brho)\phi_i\psi_i
	+ \psi_i\frac{\pa^2 F_i}{\pa\rho_i\pa q_i}(\brho)\cdot\na\phi_i
	+ \phi_i\frac{\pa^2 F_i}{\pa q_i\pa\rho_i}(\brho)\cdot\na\psi_i \\
	&\phantom{xx}{}+ \frac{\pa^2 F_i}{\pa q_i^2}:(\na\phi_i\otimes\na\psi_i)\bigg)dx \\
  &= \sum_{i=1}^n\int_{\R^3}(\phi_i,\na\phi_i)
	\begin{pmatrix} \pa^2 F_i/\pa\rho_i^2 & \pa^2 F_i/\pa\rho_i\pa q_i \\
	\pa^2 F_i/\pa\rho_i\pa q_i & \pa^2 F_i/\pa q_i^2 \end{pmatrix}
	\begin{pmatrix} \psi_i \\ \na\psi_i \end{pmatrix} dx.
\end{align*}
This finishes the proof.
\end{proof}

Next, we define the relative potential energy
\begin{align*}
  \E(\brho|\widehat\brho) &= \E(\brho) - \E(\widehat\brho)
	- \sum_{i=1}^n\bigg\langle\frac{\delta\E}{\delta\rho_i}(\widehat\brho),
	\rho_i-\widehat\rho_i\bigg\rangle.
\end{align*}
Taking $\psi_i=\phi_i = \rho_i - \widehat\rho_i$ in the above lemma leads to
the formula
\begin{align*} 
  \E(\brho|\widehat\brho)
	&= \int_{\R^3}\sum_{i=1}^n F_i(\rho_i,\na\rho_i|\widehat\rho_i,\na\widehat\rho_i)dx.
\end{align*}
We also define the total energy 
\begin{align}\label{3.deftot}
  \E_{\rm tot}(\brho,\bm{m}) 
	= \E(\brho)+ \int_{\R^3} \sum_{i=1}^n\frac{1}{2} \rho_i |v_i|^2 dx
	= \int_{\R^3}\sum_{i=1}^n\bigg(F_i(\rho_i,\na\rho_i)
	+ \frac12\rho_i|v_i|^2\bigg)dx
\end{align}
and the relative total energy
\begin{align}
	\E_{\rm tot}&(\brho,\bm{m} | \widehat\brho, \bm{\widehat m})=
	\E_{\rm tot}(\brho,\bm{m} ) - \E_{\rm tot}(\widehat\brho,\bm{\widehat m} )
	- \sum_{i=1}^n\bigg\langle
	\frac{\delta\E_{\rm tot}}{\delta\rho_i}(\widehat\brho,\bm{\widehat m}),
	\rho_i-\widehat\rho_i\bigg\rangle \nonumber \\
	& -\sum_{i=1}^n\bigg\langle
	\frac{\delta\E_{\rm tot}}{\delta m_i}(\widehat\brho,\bm{\widehat m}),
	\rho_i v_i-\widehat\rho_i\widehat v_i\bigg\rangle 
	= \int_{\R^3}\sum_{i=1}^n \bigg(F_i(\rho_i,\na\rho_i
	|\widehat\rho_i,\na\widehat\rho_i) 
	+ \frac{1}{2} \rho_i |v_i-\widehat v_i|^2 \bigg)dx. \label{eq:Etrel}
\end{align}

\subsection{Relative energy inequality}\label{sec.relent}

We compare a weak solution to the original system \eqref{3.rho}-\eqref{3.rhov} 
with a strong solution to the approximate system \eqref{3.rhoeps2}-\eqref{3.rhoveps2}
via a relative energy inequality. 
First, we make precise the notion of weak solution to the original system.

\begin{definition}[Weak and dissipative weak solutions]\label{def.weak}
A function $(\brho^\eps,\bm{v}^\eps)$ is called a 
{\em weak solution} to \eqref{3.rho}-\eqref{3.rhov} if for all $i=1,\ldots,n$,
\begin{align*}
  & 0\le \rho_i^\eps\in C^0([0,\infty);L^1(\R^3)), \quad 
	\rho_i^\eps v_i^\eps\in C^0([0,\infty);L^1(\R^3;\R^3)), \\
  & \rho_i^\eps v_i^\eps\otimes v_i^\eps,\ H_i^\eps\in 
	L_{\rm loc}^1([0,\infty)\times\R^3;\R^{3\times 3}), \\
	& s_i^\eps\in L^1_{\rm loc}([0,\infty)\times\R^3), \quad
	r_i^\eps\in L_{\rm loc}^1([0,\infty)\times\R^3;\R^3),
\end{align*}
and $(\brho^\eps,\bm{v}^\eps)$ solves for 
$\psi_i\in C^\infty_0([0,\infty);C^\infty(\R^3))$
and $\phi_i\in C^{2}_0([0,\infty);C^\infty(\R^3;\R^3))$,
\begin{align*}
  & -\int_0^\infty\int_{\R^3}\sum_{i=1}^n(\rho_i^\eps\pa_t\psi_i + \rho_i^\eps v_i^\eps
	\cdot\na\psi_i)dxdt = \int_{\R^3}\sum_{i=1}^n\rho_i^\eps(x,0)\psi_i(x,0)dx, \\
	& -\int_0^\infty\int_{\R^3}\sum_{i=1}^n\big(\rho_i^\eps v_i^\eps\cdot\pa_t\phi_i
	+ \rho_i^\eps v_i^\eps \otimes v_i^\eps : \na \phi_i 
    + s_i^\eps\diver\phi_i + r_i^\eps\cdot\na\diver\phi_i
	+ H_i^\eps:\na\phi_i\big)dxdt \\
	&\phantom{xxxx}{}= \int_{\R^3}(\rho_i^\eps v_i^\eps)(x,0)\cdot\phi_i(x,0)dx
	- \frac{1}{\eps}\int_0^\infty\int_{\R^3}\sum_{i,j=1}^n
	b_{ij}\rho_i^\eps\rho_j^\eps(v_i^\eps-v_j^\eps)\cdot\phi_i dxdt.
\end{align*}
Moreover, if additionally 
$\sum_{i=1}^n(F_i(\rho_i^\eps,\na\rho_i^\eps)+\frac12\rho_i^\eps|v_i^\eps|^2) 
\in C^0([0,\infty);L^1(\R^3))$ and the integrated energy inequality
\begin{align}
  -\int_0^\infty\E_{\rm tot}(\brho^\eps(t),\bm{m}^\eps(t))\theta'(t)dt
	&+ \frac{1}{2\eps}\int_0^\infty\int_{\R^3}\sum_{i,j=1}^n 
	b_{ij}\rho_i^\eps\rho_j^\eps|v_i^\eps-v_j^\eps|^2\theta(t)dxdt \nonumber \\
	&\le \E_{\rm tot}(\brho^\eps(0),\bm{m}^\eps(0))\theta(0) \label{3.dissweak}
\end{align}
holds for any $\theta\in W^{1,\infty}([0,\infty))$ compactly supported in $[0,\infty)$,
then we call $(\brho^\eps,\bm{v}^\eps)$ a {\em dissipative weak solution}.
\end{definition}

We impose the following assumption:
\begin{enumerate}[label=(\bf A\arabic*)]
\item \label{A1} The dissipative weak solution 
$(\brho^\eps,\bm{v}^\eps)$ to \eqref{3.rho}-\eqref{3.rhov} 
has finite total mass and finite total energy, i.e., for any $T>0$,
there exists a constant $K>0$ independent of $\eps$ such that
$$
  \sup_{0<t<T}\int_{\R^3}\sum_{i=1}^n\rho_i^\eps dx \le K, \quad
	\sup_{0<t<T}\int_{\R^3}\sum_{i=1}^n\bigg(F_i(\rho_i^\eps,
  \na\rho_i^\eps)+\frac12\rho_i^\eps|v_i^\eps|^2\bigg)dx \le K. 
$$
\end{enumerate}

We proceed by establishing the relative energy inequality.

\begin{proposition}[Relative energy inequality]\label{prop.rei}
Let $(\brho^\eps,\bm{v}^\eps)$ be a dissipative weak
solution to \eqref{3.rho}-\eqref{3.rhov} satisfying  \ref{A1}, 
let $(\widehat\brho^\eps,\widehat{{v}}^\eps)$ be a strong solution to  
\eqref{3.rhoeps}, \eqref{3.rhoveps}, \eqref{3.u} such that
$\widehat\rho_i^\eps>0$ in $\R^3$, $t>0$, and let assumption \ref{N} on page \pageref{N}
holds. Then
\begin{align}
  \E_{\rm tot}&(\brho^\eps,\bm{m}^\eps|\widehat\brho^\eps,\widehat{\bm{m}}^\eps)(t)
	+ \frac{1}{2\eps}\int_0^t\int_{\R^3}\sum_{i,j=1}^n b_{ij}
	´\rho_i^\eps\rho_j^\eps\big|(v_i^\eps-v_j^\eps)-(\widehat v_i^\eps-\widehat v_j^\eps)
	\big|^2 dxds \nonumber \\
	&\le \E_{\rm tot}(\brho^\eps,\bm{m}^\eps|\widehat\brho^\eps,\widehat{\bm{m}}^\eps)(0)
	- \int_0^t\int_{\R^3}\sum_{i=1}^n\rho_i^\eps
	(v_i^\eps-\widehat v_i^\eps)\otimes(v_i^\eps-\widehat v_i^\eps):\na\widehat v_i^\eps
	dxds \nonumber \\
	&\phantom{xx}{}- \int_0^t\int_{\R^3}\sum_{i=1}^n\Big(s_i(\rho_i^\eps,\na\rho_i^\eps
	|\widehat\rho_i^\eps,\na\widehat\rho_i^\eps)\diver\widehat v_i^\eps
	+ r_i(\rho_i^\eps,\na\rho_i^\eps|\widehat\rho_i^\eps,\na\widehat\rho_i^\eps) 
	\cdot \nabla \diver \widehat v^\eps\nonumber\\
	&\phantom{xx}{}+H_i(\rho_i^\eps,\na\rho_i^\eps
	|\widehat\rho_i^\eps,\na\widehat\rho_i^\eps):\na\widehat v_i^\eps\Big)dxds
	-\int_0^t\int_{\R^3}\sum_{i=1}^n\frac{\rho_i^\eps}{\widehat\rho_i^\eps}
	\widehat R_i^\eps\cdot(v_i^\eps-\widehat v_i^\eps)dxds \nonumber \\
	&\phantom{xx}{}- \frac{1}{\eps}\int_0^t\int_{\R^3}\sum_{i,j=1}^n b_{ij}\rho_i^\eps
	(\rho_j^\eps-\widehat\rho_j^\eps)(v_i^\eps-\widehat v_i^\eps)
	\cdot(\widehat v_i^\eps-\widehat v_j^\eps)dxds, \label{3.rei}
\end{align}
where $s_i$, $r_i$, and $H_i$ are defined in \eqref{3.S} and $\widehat R_i^\eps$
is defined in \eqref{3.Reps}, and the relative stresses are given by
\begin{align*}
	g_i(\rho_i^\eps,q_i^\eps|\widehat\rho_i^\eps,\widehat q_i^\eps) 
	= g_i(\rho_i^\eps,q_i^\eps) - g_i(\widehat\rho_i^\eps,\widehat q_i^\eps) 
	- \frac{\partial g_i}{\partial \rho_i} (\widehat\rho_i^\eps,\widehat q_i^\eps)
	(\rho_i^\eps - \widehat\rho_i^\eps) 
	- \frac{\partial g_i}{\partial q_i} (\widehat\rho_i^\eps,\widehat q_i^\eps)
	\cdot(q_i^\eps - \widehat q_i^\eps),
\end{align*}
where $q_i^\eps = \nabla\rho_i^\eps, \widehat q_i^\eps=\nabla \widehat \rho_i^\eps$ 
and $g_i$ represents $s_i$, $r_i$, and $H_i$.
\end{proposition}

\begin{proof}
The proof is similar to the proof of Theorem 1
in \cite{GLT17}, but we need to take care of the friction terms.
To simplify the notation, we drop the superscript $\eps$. Recall that the 
relative total energy $\E_{\rm tot}(\brho,\bm{m} | \widehat\brho, \bm{\widehat m})$ 
defined by \eqref{eq:Etrel} has four parts, $\E_{\rm tot}(\brho,\bm{m})$, 
$-\E_{\rm tot}(\widehat\brho,\bm{\widehat m})$, 
$-\sum_{i=1}^n\langle{(\delta\E_{\rm tot}}/{\delta\rho_i})
(\widehat\brho,\bm{\widehat m})$, $\rho_i-\widehat\rho_i\rangle$, and 
$-\sum_{i=1}^n\langle({\delta\E_{\rm tot}}/{\delta m_i})
(\widehat\brho,\bm{\widehat m}),
\rho_i v_i-\widehat\rho_i\widehat v_i\rangle$. We first give the energy 
inequalities for the first two terms and then use the weak formulations 
to calculate the last two terms.

{\em Step 1: The energy inequalities.}
Introducing the test function 
\begin{equation}\label{3.theta}
  \theta(s) = \left\{\begin{array}{ll}
	1 &\quad\mbox{for }0\le s<t, \\
	(t-s)/\delta+1 &\quad\mbox{for }t\le s< t+\delta, \\
	0 &\quad\mbox{for }s>t+\delta,
	\end{array}\right.
\end{equation}
in the integrated energy inequality \eqref{3.dissweak}
and passing to the limit $\delta\to 0$, we obtain 
\begin{align}
	\E_{\rm tot}(\brho(t),\bm{m}(t))
	+ \frac{1}{2\eps}\int_0^t\int_{\R^3}b_{ij}\rho_i\rho_j
	|v_i-v_j|^2 dxds \le \E_{\rm tot}(\brho(0),\bm{m }(0)). \label{eq:Energy}
\end{align}
To show the energy identity for the strong solution $(\widehat\brho^\eps,
\widehat{\bm{v}}^\eps)$, we write \eqref{3.rhoveps2} in nonconservative form:
$$
  \pa_t\widehat v_i + \widehat v_i\cdot\na\widehat v_i 
	= -\na\frac{\delta\E}{\delta\rho_i}(\widehat\brho) 
	- \frac{1}{\eps}\sum_{j=1}^n b_{ij}\widehat\rho_j
	(\widehat v_i-\widehat v_j)	+ \frac{\widehat R_i}{\widehat\rho_i}.
$$
We multiply this equation by $\widehat\rho_i\widehat v_i$, multiply \eqref{3.rhoeps2}
by $\frac12|\widehat v_i^\eps|^2$, and add the resulting equations:
\begin{equation}\label{3.aux3}
  \frac12\pa_t(\widehat\rho_i|\widehat v_i|^2)
	+ \frac12\diver(\widehat\rho_i \widehat v_i|\widehat v_i|^2)
	= -\widehat\rho_i\widehat v_i\cdot\na\frac{\delta\E}{\delta\rho_i}(\widehat\brho)
	- \frac{1}{\eps}\widehat v_i\cdot\sum_{j=1}^n b_{ij}\widehat\rho_i\widehat\rho_j
	(\widehat v_i-\widehat v_j) + \widehat v_i\cdot\widehat R_i.
\end{equation}
Furthermore, we deduce from \eqref{3.rhoeps2} that
$$
  \frac{d}{dt}\E(\widehat\brho) 
	= \sum_{i=1}^n\bigg\langle\frac{\delta\E}{\delta\rho_i}(\widehat\brho),\pa_t\widehat\rho_i
	\bigg\rangle = -\sum_{i=1}^n\bigg\langle\frac{\delta\E}{\delta\rho_i}(\widehat\brho),
	\diver(\widehat\rho_i\widehat v_i)\bigg\rangle
	= \int_{\R^3}\sum_{i=1}^n\na\frac{\delta\E}{\delta\rho_i}(\widehat\brho)
	\cdot(\widehat\rho_i\widehat v_i)dx.
$$
Integrating \eqref{3.aux3}, summing over $i=1,\ldots,n$, and inserting the 
previous identity yields
\begin{equation*}
  \frac{d}{dt}\bigg(\E(\widehat\brho) + 
	\frac12\int_{\R^3}\sum_{i=1}^n\widehat\rho_i|\widehat v_i|^2dx\bigg)
	= -\frac{1}{\eps}\int_{\R^3}\sum_{i,j=1}^n b_{ij}\widehat\rho_i\widehat\rho_j
	(\widehat v_i-\widehat v_j)\cdot\widehat v_i dx 
	+ \int_{\R^3} \sum_{i=1}^n\widehat R_i\cdot\widehat v_i dx.
\end{equation*}
The symmetry of $(b_{ij})$ and integration of the above equality over $(0,t)$ 
lead to the following energy equality:
\begin{align}\label{eq:Energy2}
	  & \E_{\rm tot}(\widehat\brho(t),\widehat{\bm{m}}(t))
	+ \frac{1}{2\eps}\int_0^t\int_{\R^3}\sum_{i,j=1}^n b_{ij}\widehat\rho_i
	\widehat\rho_j|\widehat v_i-\widehat v_j|^2 dxds \nonumber\\
	&\phantom{xx}{}= \E_{\rm tot}(\widehat\brho(0),\widehat{\bm{m}}(0)) 
	+ \int_0^t\int_{\R^3}\sum_{i=1}^n\widehat R_i\cdot\widehat v_i dxds.
\end{align}

{\em Step 2: Equation for the difference.}
We proceed to calculate 
$$
  -\sum_{i=1}^n\bigg\langle\frac{\delta\E_{\rm tot}}{\delta\rho_i}(\widehat\brho,
	\bm{\widehat m}),\rho_i-\widehat\rho_i\bigg\rangle \quad \mbox{and} \quad 
  -\sum_{i=1}^n\bigg\langle\frac{\delta\E_{\rm tot}}{\delta m_i}(\widehat\brho,\bm{\widehat m}),
	\rho_i v_i-\widehat\rho_i\widehat v_i\bigg\rangle.
$$
Following the definition of the weak solutions to \eqref{3.rho}-\eqref{3.rhov} and 
\eqref{3.rhoeps2}-\eqref{3.rhoveps2}, the differences of the solutions 
$(\rho_i-\widehat\rho_i, v_i - \widehat v_i)$ satisfy
\begin{align*}
  -\int_0^\infty&\int_{\R^3}\sum_{i=1}^n\big((\rho_i-\widehat\rho_i)\pa_s\psi_i
	+ (\rho_i v_i-\widehat\rho_i\widehat v_i)\cdot\na\psi_i\big)dxds \\
	&= \int_{\R^3}\sum_{i=1}^n\big(\rho_i(x,0)-\widehat\rho_i(x,0)\big)\psi_i(x,0)dx, \\
	-\int_0^\infty&\int_{\R^3}\sum_{i=1}^n\big((\rho_i v_i-\widehat\rho_i\widehat v_i)
	\cdot\pa_s\phi_i 
	+ (\rho_i v_i\otimes v_i - \widehat\rho_i\widehat v_i\otimes\widehat v_i):\na\phi_i \\
	&\phantom{xx}{}+ (s_i-\widehat s_i)\diver\phi_i + (H_i-\widehat H_i):\na\phi_i
	+ (r_i-\widehat r_i):\na\diver\phi_i\big)dxds \\
	&= \int_{\R^3}\sum_{i=1}^n((\rho_i v_i)(x,0) 
	- (\widehat\rho_i\widehat v_i)(x,0))\phi_i(x,0)dx \\
	 &\phantom{xx}{}-\frac{1}{\eps}\int_0^\infty\int_{\R^3}
	\sum_{i,j=1}^n b_{ij}\big(\rho_i\rho_j
	(v_i-v_j) - \widehat\rho_i\widehat\rho_j(\widehat v_i-\widehat v_j)\big)
	\cdot\phi_i dxds \\
	&\phantom{xx}{}- \int_0^\infty\int_{\R^3}\sum_{i=1}^n\widehat R_i\cdot\phi_i dxds,
\end{align*}
where $s_i=s_i(\rho_i,\na\rho_i)$, $\widehat s_i=s_i(\widehat\rho_i,\na\widehat\rho_i)$,
and similar for the other quantities. Taking the test functions
$$
  \psi_i(s) = \theta(s)\bigg(\frac{\pa\widehat F_i}{\pa\rho_i}
	- \diver\frac{\pa\widehat F_i}{\pa q_i} - \frac12|\widehat v_i|^2\bigg)(s), \quad
	\phi_i(s) = \theta(s)\widehat v_i(s),
$$
where $\theta$ is defined in \eqref{3.theta} and 
$\widehat F_i=F_i(\widehat\rho_i,\na\widehat\rho_i)$, the sum of the above 
equations becomes
\begin{align}
  \sum_{i=1}^n & \left(\left\langle\frac{\delta\E_{\rm tot}}{\delta\rho_i}
	(\widehat\brho,\bm{\widehat m}),\rho_i-\widehat\rho_i\right\rangle 
	+ \left\langle\frac{\delta\E_{\rm tot}}{\delta m_i}(\widehat\brho,\bm{\widehat m}),
	\rho_i v_i-\widehat\rho_i\widehat v_i\right\rangle\right) \bigg|_0^t \nonumber \\
  &= \int_{\R^3}\sum_{i=1}^n\bigg(\frac{\pa\widehat F_i}{\pa\rho_i}
	- \diver\frac{\pa\widehat F_i}{\pa q_i} - \frac12|\widehat v_i|^2\bigg)
	(\rho_i-\widehat\rho_i)\bigg|_{0}^t dx
	+ \int_{\R^3}\sum_{i=1}^n(\rho_i v_i-\widehat\rho_i\widehat v_i)\cdot\widehat v_i
	\bigg|_{0}^t dx \nonumber \\
	&= \int_0^t\int_{\R^3}\sum_{i=1}^n\bigg\{\pa_s\bigg(\frac{\pa\widehat F_i}{\pa\rho_i}
	- \diver\frac{\pa\widehat F_i}{\pa q_i} - \frac12|\widehat v_i|^2\bigg)
	(\rho_i-\widehat\rho_i) \nonumber  \\
	&\phantom{xx}{}+ (\rho_i v_i-\widehat\rho_i\widehat v_i)\cdot\na
	\bigg(\frac{\pa\widehat F_i}{\pa\rho_i}
	- \diver\frac{\pa\widehat F_i}{\pa q_i} - \frac12|\widehat v_i|^2\bigg)\bigg\}dxds 
	\nonumber \\
	&\phantom{xx}{}
	+ \int_0^t\int_{\R^3}\sum_{i=1}^n\big((\rho_i v_i-\widehat\rho_i\widehat v_i)
	\cdot\pa_s\widehat v_i + (\rho_i v_i\otimes v_i - \widehat\rho_i\widehat v_i\otimes
	\widehat v_i):\na\widehat v_i\big) dxds \nonumber  \\
	&\phantom{xx}{}+ \int_0^t\int_{\R^3}\sum_{i=1}^n\big(
	(s_i-\widehat s_i)\diver \widehat v_i
	+ (H_i-\widehat H_i):\na\widehat v_i + (r_i-\widehat r_i)\cdot\na\diver\widehat v_i
	\big)dxds \nonumber \\
	&\phantom{xx}{}
	- \frac{1}{\eps}\int_0^t\int_{\R^3}\sum_{i,j=1}^n b_{ij}\big(\rho_i\rho_j(v_i-v_j)
	- \widehat\rho_i\widehat\rho_j(\widehat v_i-\widehat v_j)\big)\cdot\widehat v_i
	dxds \nonumber \\
	&\phantom{xx}{}- \int_0^t\int_{\R^3}\sum_{i=1}^n\widehat R_i\cdot\widehat 
	v_i dxds \nonumber\\
	&=: I_1 + I_2+I_3+I_4+I_5.
	\label{3.aux}
\end{align}

We reorganize the term $I_1$ as follows:
\begin{align*}
	I_1 &= I_{11} + I_{12} + I_{13}, \quad\mbox{where} \\
	 I_{11} &= \int_0^t\int_{\R^3}\sum_{i=1}^n\pa_s\bigg(
	\frac{\pa\widehat F_i}{\pa\rho_i}	- \diver\frac{\pa\widehat F_i}{\pa q_i}\bigg)
	(\rho_i-\widehat\rho_i)dxds, \\
	I_{12} &= \int_0^t\int_{\R^3}\sum_{i=1}^n(\rho_i v_i-\widehat\rho_i\widehat v_i)
	\cdot\na\bigg(\frac{\pa\widehat F_i}{\pa\rho_i}
	- \diver\frac{\pa\widehat F_i}{\pa q_i}\bigg)dxds, \\
	I_{13} &= \int_0^t\int_{\R^3}\sum_{i=1}^n\bigg(-\frac12\pa_s\big(|\widehat v_i|^2\big)
	(\rho_i-\widehat\rho_i)
	-\frac12 (\rho_i v_i-\widehat\rho_i\widehat v_i)
	\cdot\na\big(|\widehat v_i|^2) \bigg) dxds, \\
\end{align*}

{\em Step 3: Calculation of $I_{11}$ and $I_{12}$.}
Using \eqref{3.rhoeps2}, we obtain:
\begin{align}
  I_{11}	&= \int_0^t\int_{\R^3}\sum_{i=1}^n
	\bigg\{\bigg(\frac{\pa^2\widehat F_i}{\pa\rho_i^2}\pa_s\widehat\rho_i
	+ \frac{\pa^2\widehat F_i}{\pa\rho_i\pa q_i}\cdot\pa_s\na\widehat\rho_i\bigg)
	(\rho_i-\widehat\rho_i) \nonumber \\
	&\phantom{xx}{}- \bigg(\diver \bigg(\frac{\pa^2\widehat F_i}{\pa q_i\pa\rho_i}
	\pa_s\widehat\rho_i \bigg)
	+ \diver\bigg(\frac{\pa^2\widehat F_i}{\pa q_i^2}\cdot\pa_s\na\widehat\rho_i\bigg)
	\bigg)(\rho_i-\widehat\rho_i)\bigg\}dxds \nonumber \\
	&= -\int_0^t\int_{\R^3}\sum_{i=1}^n\bigg\{\bigg(\frac{\pa^2\widehat F_i}{\pa\rho_i^2}
	\diver(\widehat\rho_i\widehat v_i)
	+ \frac{\pa^2\widehat F_i}{\pa\rho_i\pa q_i}\na\diver(\widehat\rho_i\widehat v_i)
	\bigg)(\rho_i-\widehat\rho_i) \nonumber \\
	&\phantom{xx}{}- \bigg(\diver\bigg(\frac{\pa^2\widehat F_i}{\pa q_i\pa\rho_i}
	\diver(\widehat\rho_i\widehat v_i)\bigg)
	+ \diver\bigg(\frac{\pa^2\widehat F_i}{\pa q_i^2}\cdot\na
	\diver(\widehat\rho_i\widehat v_i)\bigg)\bigg)(\rho_i-\widehat\rho_i)\bigg\}dxds 
	\nonumber \\
	&= -\int_0^t\int_{\R^3}\sum_{i=1}^n\bigg(\frac{\pa^2\widehat F_i}{\pa\rho_i^2}
	\diver(\widehat\rho_i\widehat v_i)(\rho_i-\widehat\rho_i)
	+ \frac{\pa^2\widehat F_i}{\pa\rho_i\pa q_i}\cdot\na\diver(\widehat\rho_i\widehat v_i)
	(\rho_i-\widehat\rho_i) \nonumber \\
	&\phantom{xx}{}+ \frac{\pa^2\widehat F_i}{\pa q_i\pa\rho_i}\cdot
	\na(\rho_i-\widehat\rho_i)\diver(\widehat\rho_i\widehat v_i)
	+ \frac{\pa^2\widehat F_i}{\pa q_i^2}:\big(\na\diver(\widehat\rho_i\widehat v_i)
	\otimes\na(\rho_i-\widehat\rho_i)\big)\bigg)dxds. \label{3.first}
\end{align}
We claim that the second-order derivatives of $F_i$ can be related to
the functional derivative of $S_i$. Indeed, we take the variational derivative
of the weak formulation of \eqref{3.ES},
$$
  \bigg\langle\frac{\delta\E}{\delta\rho_i}(\widehat\brho),
	\diver(\widehat\rho_i\phi_i)\bigg\rangle
  = -\int_{\R^3}\widehat\rho_i\na\frac{\delta\E}{\delta\rho_i}(\widehat\brho)\cdot
	\phi_i dx = -\int_{\R^3}S_i(\widehat\brho_i):\na\phi_i dx
$$
for some test function $\phi_i$. Let $\bm\psi=(\psi_1,\ldots,\psi_n)$ 
be another test function. Then the limit $\tau\to 0$ in
\begin{align*}
  \frac{1}{\tau}\bigg\langle\frac{\delta\E}{\delta\rho_i}
	(\widehat\brho+\tau\bm\psi)	&- \frac{\delta\E}{\delta\rho_i}(\widehat\brho),
	\diver(\widehat\rho_i\phi_i)\bigg\rangle
	+ \frac{1}{\tau}\bigg\langle\frac{\delta\E}{\delta\rho_i}
	(\widehat\brho+\tau\bm\psi),\diver\big((\widehat\rho_i+\tau\psi_i)\phi_i\big)
	-\diver(\widehat\rho_i\phi_i)\bigg\rangle \\
	&= -\frac{1}{\tau}\int_{\R^3}\big(S_i(\widehat\rho_i+\tau\psi_i)
	- S_i(\widehat\rho_i)\big):\na\phi_i dx
\end{align*}
and summation over $i=1,\ldots,n$ leads to
\begin{align*}
  \sum_{i=1}^n &
	\bigg\langle\!\!\!\bigg\langle\frac{\delta^2\E}{\delta\rho_i^2}(\widehat\brho),
	\big(\diver(\widehat\rho_i\phi_i),\psi_i\big)\bigg\rangle\!\!\!\bigg\rangle
	+ \sum_{i=1}^n
	\bigg\langle\frac{\delta\E}{\delta\rho_i}(\widehat\brho),\diver(\psi_i\phi_i)
	\bigg\rangle \\
	&= -\sum_{i=1}^n\int_{\R^3}\bigg\langle\frac{\delta S_i}{\delta\rho_i}(\widehat\brho),
	\psi_i\bigg\rangle:\na\phi_i dx.
\end{align*}
Inserting the expressions for the variational derivatives from Lemma \ref{lem.second}
and choosing $\phi_i=\widehat v_i$ and $\psi_i=\rho_i-\widehat\rho_i$, we deduce that
\begin{align}
  \int_{\R^3}&\sum_{i=1}^n\bigg(\frac{\pa^2\widehat F_i}{\pa\rho_i^2}
	\diver(\widehat\rho_i\widehat v_i)(\rho_i-\widehat\rho_i) 
	+ \frac{\pa^2\widehat F_i}{\pa\rho_i\pa q_i}
	\cdot\na(\diver(\widehat\rho_i\widehat v_i))(\rho_i-\widehat\rho_i) \nonumber \\
  &\phantom{xx}{}
	+ \frac{\pa^2\widehat F_i}{\pa q_i\pa\rho_i}\cdot\na(\rho_i-\widehat\rho_i)
	\diver(\widehat\rho_i\widehat v_i)
	+ \frac{\pa^2\widehat F_i}{\pa q_i^2}:\big(\na(\diver(\widehat\rho_i\widehat v_i))
	\otimes\na(\rho_i-\widehat\rho_i)\big)\bigg)dx \nonumber \\
  &\phantom{xx}{}
	- \int_{\R^3}\sum_{i=1}^n\na\bigg(\frac{\pa\widehat F_i}{\pa\rho_i}
	- \diver\frac{\pa\widehat F_i}{\pa q_i}\bigg)\cdot((\rho_i-\widehat\rho_i)
	\widehat v_i)dx \nonumber \\
	&= \int_{\R^3}\sum_{i=1}^n\bigg\{\bigg(\frac{\pa\widehat s_i}{\pa\rho_i}
	(\rho_i-\widehat\rho_i)
	+ \frac{\pa\widehat s_i}{\pa q_i}\cdot\na(\rho_i-\widehat\rho_i)\bigg)
	\diver\widehat v_i \nonumber \\
  &\phantom{xx}{}
	+ \bigg(\frac{\pa\widehat r_i}{\pa\rho_i}(\rho_i-\widehat\rho_i)
	+ \frac{\pa\widehat r_i}{\pa q_i}\cdot\na(\rho_i-\widehat\rho_i)\bigg)
	\cdot\na\diver\widehat v_i \nonumber \\
	&\phantom{xx}{}+ \bigg(\frac{\pa\widehat H_i}{\pa\rho_i}(\rho_i-\widehat\rho_i)
	+ \frac{\pa\widehat H_i}{\pa q_i}\cdot\na(\rho_i-\widehat\rho_i)\bigg)
	:\na\widehat v_i\bigg\}dx. \label{3.T1}
\end{align}
The first four terms on the left-hand side correspond, up to the sign,
to the right-hand side of \eqref{3.first}. Using
\begin{align*}
	-\int_{\R^3}&\sum_{i=1}^n\na\bigg(\frac{\pa\widehat F_i}{\pa\rho_i}
	- \diver\frac{\pa\widehat F_i}{\pa q_i}\bigg)\cdot((\rho_i-\widehat\rho_i)
	\widehat v_i)dx +\int_{\R^3}\sum_{i=1}^n(\rho_i v_i-\widehat\rho_i\widehat v_i)
	\cdot\na\bigg(\frac{\pa\widehat F_i}{\pa\rho_i}
	- \diver\frac{\pa\widehat F_i}{\pa q_i}\bigg)dx \\
	&= \int_{\R^3}\sum_{i=1}^n \na\bigg(\frac{\pa\widehat F_i}{\pa\rho_i}
	- \diver\frac{\pa\widehat F_i}{\pa q_i}\bigg) \cdot \rho_i(v_i - \widehat v_i) dx,
 \end{align*} 
we find that
\begin{align}
  I_{11} + I_{12} &= \int_0^t\int_{\R^3}\sum_{i=1}^n 
	\na\bigg(\frac{\pa\widehat F_i}{\pa\rho_i}
	- \diver\frac{\pa\widehat F_i}{\pa q_i}\bigg) 
	\cdot \rho_i(v_i - \widehat v_i) dxds \nonumber \\
	&\phantom{xx}{}- \int_0^t\int_{\R^3}\sum_{i=1}^n
	\bigg\{\bigg(\frac{\pa\widehat s_i}{\pa\rho_i}(\rho_i-\widehat\rho_i)
	+ \frac{\pa\widehat s_i}{\pa q_i}\cdot\na(\rho_i-\widehat\rho_i)\bigg)
	\diver\widehat v_i \nonumber \\
  &\phantom{xx}{}
	+ \bigg(\frac{\pa\widehat r_i}{\pa\rho_i}(\rho_i-\widehat\rho_i)
	+ \frac{\pa\widehat r_i}{\pa q_i}\cdot\na(\rho_i-\widehat\rho_i)\bigg)
	\cdot\na\diver\widehat v_i \nonumber \\
	&\phantom{xx}{}+ \bigg(\frac{\pa\widehat H_i}{\pa\rho_i}(\rho_i-\widehat\rho_i)
	+ \frac{\pa\widehat H_i}{\pa q_i}\cdot\na(\rho_i-\widehat\rho_i)\bigg)
	:\na\widehat v_i\bigg\}dxds. \label{3.I1I2}
\end{align}

{\em Step 4: Calculation of $I_{13}$ and $I_2$.}
The sum of $I_{13}$ and $I_2$ is
\begin{align}
	I_{13} + I_2 &= \int_0^t\int_{\R^3}\sum_{i=1}^n
	\bigg(-\frac12\pa_s\big(|\widehat v_i|^2\big)(\rho_i-\widehat\rho_i)
	-\frac12 (\rho_i v_i-\widehat\rho_i\widehat v_i)
	\cdot\na\big(|\widehat v_i|^2) \bigg) dxds \nonumber\\
	&\phantom{xx}{}+ \int_0^t\int_{\R^3}\sum_{i=1}^n
	\big((\rho_i v_i-\widehat\rho_i\widehat v_i)
	\cdot\pa_s\widehat v_i + (\rho_i v_i\otimes v_i - \widehat\rho_i\widehat v_i\otimes
	\widehat v_i):\na\widehat v_i\big) dxds \nonumber  \\
	& = \int_0^t\int_{\R^3}\sum_{i=1}^n\big(- \widehat v_i \otimes (\rho_i v_i 
	- \widehat \rho_i \widehat v_i) + (\rho_iv_i\otimes v_i -\widehat\rho_i 
	\widehat v_i \otimes \widehat v_i)\big):\na\widehat v_i dxds \nonumber \\
	&\phantom{xx}{}+\int_0^t\int_{\R^3}\sum_{i=1}^n\rho_i(v_i-\widehat v_i) 
	\pa_s\widehat v_i dxds. \label{eq:I13I2}
\end{align}

Observing that \eqref{3.rhoveps2} reads in nonconservative form as
$$
  \pa_t\widehat v_i + \widehat v_i\cdot\na\widehat v_i
	= - \na\bigg(\frac{\pa\widehat F_i}{\pa\rho_i}
	- \diver\frac{\pa\widehat F_i}{\pa q_i}\bigg)  
	- \frac{1}{\eps}\sum_{j=1}^n b_{ij}\widehat\rho_j(\widehat v_i-\widehat v_j) 
	+ \frac{\widehat R_i}{\widehat\rho_i},
$$
it follows that
\begin{align}
 \int_0^t&\int_{\R^3}\sum_{i=1}^n
	\rho_i(v_i-\widehat v_i)\cdot\pa_s\widehat v_i dxds \nonumber \\
	&= \int_0^t\int_{\R^3}\sum_{i=1}^n
	\rho_i(v_i-\widehat v_i)\cdot\bigg(-\widehat v_i\cdot\na\widehat v_i
	- \na\bigg(\frac{\pa\widehat F_i}{\pa\rho_i}
	- \diver\frac{\pa\widehat F_i}{\pa q_i}\bigg)\cdot(\rho_i(v_i-\widehat v_i)) 
	\nonumber\\
	&\phantom{xx}{}- \frac{1}{\eps}\sum_{j=1}^n b_{ij}\widehat\rho_j
	(\widehat v_i-\widehat v_j) 
	+ \frac{\widehat R_i}{\widehat\rho_i}\bigg)dxds \nonumber \\
	&= \int_0^t\int_{\R^3}\sum_{i=1}^n\bigg(-\rho_i(v_i-\widehat v_i)\otimes\widehat v_i
	:\na\widehat v_i 
	-\na\bigg(\frac{\pa\widehat F_i}{\pa\rho_i}
	- \diver\frac{\pa\widehat F_i}{\pa q_i}\bigg)  \nonumber\\
	&\phantom{xx}{}- \frac{1}{\eps}\sum_{j=1}^n b_{ij}\rho_i\widehat\rho_j
	(v_i-\widehat v_i)\cdot(\widehat v_i-\widehat v_j) 
	+ \frac{\rho_i}{\widehat\rho_i}(v_i-\widehat v_i)\cdot \widehat R_i\bigg)dxds.
	\label{3.I3}
\end{align}

Substituting the above formula into \eqref{eq:I13I2} leads to
\begin{align}
	I_{13}+I_2 &=\int_0^t\int_{\R^3}\sum_{i=1}^n \rho_i(v_i-\widehat v_i)\otimes
  (v_i-\widehat v_i):\na\widehat v_i dxds \nonumber \\
   &\phantom{xx}{}- \int_0^t\int_{\R^3}\sum_{i=1}^n
	\na\bigg(\frac{\pa\widehat F_i}{\pa\rho_i}
	- \diver\frac{\pa\widehat F_i}{\pa q_i}\bigg) 
	\cdot \rho_i(v_i-\widehat v_i)dxds \nonumber\\
	&\phantom{xx}{}-\frac{1}{\eps}\int_0^t\int_{\R^3}\sum_{i,j=1}^n 
	b_{ij}\rho_i\widehat\rho_j
	(v_i-\widehat v_i)\cdot(\widehat v_i-\widehat v_j) dxds \nonumber \\
	&\phantom{xx}{}+ \int_0^t\int_{\R^3}\sum_{i=1}^n  
	\frac{\rho_i}{\widehat\rho_i}(v_i-\widehat v_i)\cdot \widehat R_idxds. 
	\label{eq:I13I2sum}
\end{align}

{\em Step 5: Calculation of $I_4$.}
We collect the terms in $I_4$ and the friction term in \eqref{eq:I13I2sum}:
\begin{align}
  \frac{1}{\eps}\sum_{i,j=1}^n &\Big(\big(-b_{ij}\rho_i\rho_j(v_i-v_j)
	+ \widehat\rho_i\widehat\rho_j(\widehat v_i-\widehat v_j)\big)\cdot\widehat v_i
	- b_{ij}\rho_i\widehat\rho_j(v_i-\widehat v_i)\cdot(\widehat v_i-\widehat v_j)\Big) 
	\nonumber \\
  &= \frac{1}{\eps}\sum_{i,j=1}^n b_{ij}\rho_i\rho_j(v_i-v_j)\cdot(v_i-\widehat v_i)
  - \frac{1}{\eps}\sum_{i,j=1}^n b_{ij}\rho_i\rho_j(v_i-v_j)\cdot v_i \nonumber \\
	&\phantom{xx}{}
	+ \frac{1}{\eps}\sum_{i,j=1}^n b_{ij}\widehat\rho_i\widehat\rho_j
	(\widehat v_i-\widehat v_j)\cdot\widehat v_i 
	- \frac{1}{\eps}\sum_{i,j=1}^n b_{ij}\rho_i\widehat\rho_j(v_i-\widehat v_i)
	\cdot(\widehat v_i-\widehat v_j). \label{3.auxbij}
\end{align}
By the symmetry of $(b_{ij})$, the second and the third term on the right-hand side 
become
\begin{align*}
  -\frac{1}{\eps}\sum_{i,j=1}^n b_{ij}\rho_i\rho_j(v_i-v_j)\cdot v_i
	&= -\frac{1}{2\eps}\sum_{i,j=1}^n b_{ij}\rho_i\rho_j|v_i-v_j|^2, \\
	\frac{1}{\eps}\sum_{i,j=1}^n b_{ij}\widehat\rho_i\widehat\rho_j
	(\widehat v_i-\widehat v_j)\cdot\widehat v_i
	&= \frac{1}{2\eps}\sum_{i,j=1}^n b_{ij}\widehat\rho_i\widehat\rho_j
	|\widehat v_i-\widehat v_j|^2.
\end{align*}
We write the last term on the right-hand side of \eqref{3.auxbij} as
\begin{align*}
  -\frac{1}{\eps}\sum_{i,j=1}^n b_{ij}\rho_i\widehat\rho_j(v_i-\widehat v_i)\cdot
	(\widehat v_i-\widehat v_j)
	&= \frac{1}{\eps}\sum_{i,j=1}^n b_{ij}\rho_i(\rho_j-\widehat\rho_j)
	(v_i-\widehat v_i)\cdot(\widehat v_i-\widehat v_j) \\
	&\phantom{xx}{}- \frac{1}{\eps}\sum_{i,j=1}^n b_{ij}\rho_i\rho_j
	(v_i-\widehat v_i)\cdot(\widehat v_i-\widehat v_j). 
\end{align*}
The last term can be combined with the first term on the right-hand side
of \eqref{3.auxbij}:
\begin{align*}
  \frac{1}{\eps}&\sum_{i,j=1}^n b_{ij}\rho_i\rho_j(v_i-v_j)\cdot(v_i-\widehat v_i)
	- \frac{1}{\eps}\sum_{i,j=1}^n b_{ij}\rho_i\rho_j(v_i-\widehat v_i)\cdot
	(\widehat v_i-\widehat v_j) \\
	&= \frac{1}{\eps}\sum_{i,j=1}^n b_{ij}\rho_i\rho_j
	\big((v_i-v_j)-(\widehat v_i-\widehat v_j)\big)\cdot(v_i-\widehat v_i) \\
	&= \frac{1}{2\eps}\sum_{i,j=1}^n b_{ij}\rho_i\rho_j
	|(v_i-v_j)-(\widehat v_i-\widehat v_j)|^2.
\end{align*}
Then, combining these results, we conclude from \eqref{3.auxbij} that
\begin{align}
  I_4 &- \frac{1}{\eps}\int_0^t\int_{\R^3}\sum_{i,j=1}^n b_{ij}\rho_i\widehat\rho_j
	(v_i-\widehat v_i)\cdot(\widehat v_i-\widehat v_j)dxds \nonumber \\
  &= -\frac{1}{\eps}\int_0^t\int_{\R^3}
	\sum_{i,j=1}^n \Big(\big(b_{ij}\rho_i\rho_j(v_i-v_j)
	- \widehat\rho_i\widehat\rho_j(\widehat v_i-\widehat v_j)\big)\cdot\widehat v_i 
	\nonumber \\
	&\phantom{xx}{}
	+ b_{ij}\rho_i\widehat\rho_j(v_i-\widehat v_i)\cdot
	(\widehat v_i-\widehat v_j)\Big) dxds \nonumber \\
	&= \frac{1}{2\eps}\int_0^t\int_{\R^3}\sum_{i,j=1}^n b_{ij}\rho_i\rho_j
	|(v_i-v_j)-(\widehat v_i-\widehat v_j)|^2 dxds \nonumber \\
	&\phantom{xx}{}- \frac{1}{2\eps}\int_0^t\int_{\R^3}
	\sum_{i,j=1}^n b_{ij}\rho_i\rho_j|v_i-v_j|^2 
	+ \frac{1}{2\eps}\int_0^t\int_{\R^3}
	\sum_{i,j=1}^n b_{ij}\widehat\rho_i\widehat\rho_j |\widehat v_i-\widehat v_j|^2 dxds
	\nonumber \\
	&\phantom{xx}{}
	+\frac{1}{\eps}\int_0^t\int_{\R^3}\sum_{i,j=1}^n b_{ij}\rho_i(\rho_j-\widehat\rho_j)
	(v_i-\widehat v_i)\cdot(\widehat v_i-\widehat v_j)dxds. \label{3.I7}
\end{align}

Finally, we insert \eqref{3.I1I2}, \eqref{eq:I13I2sum}, and
\eqref{3.I7} into \eqref{3.aux} and then subtract the resulting \eqref{3.aux} 
and equation \eqref{eq:Energy2}
from \eqref{eq:Energy} to arrive at \eqref{3.rei}.
\end{proof}


\subsection{Convergence of the Chapman-Enskog expansion }\label{sec.conv}

We proceed to justify the Chapman-Enskog expansion using the relative entropy identity. 
We place a series of assumptions:
\begin{enumerate}[label=\bf (A\arabic*)]
  \setcounter{enumi}{1}
\item \label{A2} The strong solution $(\widehat \rho_i^\eps, \widehat v^\eps)$ to 
\eqref{3.rhoeps}-\eqref{3.rhoveps} satisfies for $\widehat v_i^\eps 
= \widehat v^\eps + \widehat u_i^\eps$ with $\widehat u_i^\eps$ being
a solution of \eqref{3.u}: 
There exists a constant $C>0$ such that for all $\eps>0$ and $i=1,\ldots,n$,
$$
  \|\nabla \widehat v_i^\eps\|_{L^\infty([0,T];L^\infty(\R^3))}
	+ \|\na\diver \widehat v_i^\eps\|_{L^\infty([0,T];L^\infty(\R^3))} \le C.
$$
\item \label{A3} The strong solution $\widehat\rho_i^\eps$ to 
\eqref{3.rhoeps}-\eqref{3.rhoveps} satisfies: There are constants $K>\kappa>0$ 
such that for all $\eps>0$, $x\in\R^3$, $t\in(0,T)$, and $i=1,\ldots,n$,
$$
  \kappa\le\widehat\rho_i^\eps(x,t)\le K .
$$
\item \label{A4} Let $F_i(\rho_i,q_i)=h_i(\rho_i) + \frac12\kappa_i(\rho_i)|q_i|^2$,
where $h_i$ and $\kappa_i$ are $C^3$ functions and 
there exists a constant $\alpha>0$ such that
for all $i=1,\ldots,n$ and $\rho_i\ge 0$, 
$$
  h_i''(\rho_i)\ge\alpha, \quad 
	\kappa_i(\rho_i){\kappa_i''(\rho_i)} -  {2\kappa_i'(\rho_i)^2} \ge 0, 
	\quad \kappa_i (\rho_i)>0.
$$
\item \label{A5} The dissipative weak solution 
$(\brho^\eps,\bm{v}^\eps)$ satisfies that $\rho_i^\eps$ are uniformly bounded in 
$L^\infty([0,T];L^\infty(\mathbb{R}^3))$ and there are constants $K>\kappa>0$ such that 
$$
  \kappa \le \rho_i^\eps \le K\quad \mbox{in }\R^3,\ 0<t<T.
$$
\end{enumerate}

Hypothesis \ref{A1} concerns the family of dissipative weak solutions
which is assumed to satisfy the uniform bounds \ref{A5}. Hypotheses \ref{A2} and 
\ref{A3} concern the family of strong solutions to
the target system \eqref{3.rhoeps}-\eqref{3.rhoveps}.

Hypothesis \ref{A4} is a structural hypothesis on the model.  
It is in particular satisfied for $\kappa_i(\rho_i)=\rho_i^s$ with
$s \in [-1,0]$ for $\rho_i>0$. The important special cases $s = -1$ 
(corresponding to the quantum hydrodynamic system) and $s=0$ 
(corresponding to constant capillarity) are included.

\begin{theorem}\label{thm.convCE}
Let $(\brho^\eps, \bm{v}^\eps)$ be a dissipative weak
solution to \eqref{3.rho}-\eqref{3.rhov} satisfying assumption \ref{A1} and 
\ref{A5}, and let $(\widehat\brho^\eps,\widehat v^\eps)$
be a strong solution to \eqref{3.rhoeps}-\eqref{3.rhoveps} satisfying
assumptions \ref{A2}-\ref{A4}. Furthermore, let assumption \ref{N} on page \pageref{N}
hold and let $T>0$. We introduce
$$
  \chi(t) = \int_{\R^3}\sum_{i=1}^n\bigg(
	\frac12\rho_i^\eps|v_i^\eps-\widehat v_i^\eps|^2 
	+  (\rho_i^\eps-\widehat\rho_i^\eps)^2 
	+ \frac{1}{2\kappa_i(\rho_i^\eps)}|\kappa_i(\rho_i^\eps)\nabla \rho_i^\eps 
	- \kappa_i(\widehat \rho_i^\eps)\na\widehat \rho_i^\eps|^2\bigg)(t)dx.
$$
Then there exists a constant $C>0$ such that for all $\eps>0$ and $t\in(0,T)$,
$$
  \chi(t) \le C(\chi(0)+\eps^2), \quad t\in(0,T).
$$
In particular, if $\chi(0)\to 0$ as $\eps\to 0$, we have
$$
  \sup_{t\in(0,T)}\chi(t) \to 0\quad\mbox{as }\eps\to 0.
$$
\end{theorem}

\begin{proof}
We apply the relative energy inequality \eqref{3.rei}. First, we relate
the total relative entropy to $\chi(t)$. 
The superscript $\eps$ is dropped for simplicity of calculations.
The relative potential is 
\begin{align}\label{eq:relativeF}
	F_i(\rho_i,q_i|\widehat \rho_i,\widehat q_i)
	&= F_i(\rho_i,q_i) - F_i(\widehat \rho_i,\widehat q_i)
	- \frac{\pa F_i}{\pa\rho_i}(\widehat \rho_i,\widehat q_i)(\rho_i-\widehat \rho_i)
	- \frac{\pa F_i}{\pa q_i}(\widehat\rho_i,\widehat q_i)\cdot(q_i-\widehat q_i) \\
	& = h_i(\rho_i|\widehat \rho_i) + \bigg(\frac{1}{2} \kappa_i(\rho_i) 
	|q_i|^2\bigg)(\rho_i,q_i|\widehat \rho_i,\widehat q_i). \nonumber
\end{align}
The second term on the right-hand side of the above equation is calculated in 
detail as follows:
\begin{align}
    	\bigg(\frac{1}{2} &\kappa_i(\rho_i) |q_i|^2\bigg) 
			(\rho_i,q_i|\widehat \rho_i,\widehat q_i) \nonumber	\\
    	&= \frac{1}{2} \kappa_i(\rho_i) |q_i|^2 - \frac{1}{2} \kappa_i(\widehat \rho_i) 
			|\widehat q_i|^2 - \frac{1}{2} \kappa_i'(\widehat \rho_i) |\widehat q_i|^2 
			(\rho_i - \widehat\rho_i)
    	 - \kappa_i(\widehat\rho_i) \widehat q_i(q_i-\widehat q_i) \nonumber\\
    	&= \frac{1}{2\kappa_i(\rho_i)} (\kappa_i^2(\rho_i)|q_i|^2
			- 2\kappa_i(\widehat \rho_i) \kappa_i(\rho_i) q_i \cdot \widehat q_i
			+ \kappa_i^2(\widehat \rho_i)|\widehat q_i|^2) \nonumber\\
    	&\phantom{xx}{}+ \frac{1}{2} |\widehat q_i|^2
			\bigg(-\frac{\kappa_i^2(\widehat\rho_i)}{\kappa_i(\rho_i)}
			+ \kappa_i(\widehat \rho_i) - \kappa_i'(\widehat\rho_i)(\rho_i - \widehat \rho_i) 
			\bigg) \nonumber\\
    	&= \frac{1}{2\kappa_i(\rho_i)}|\kappa_i(\rho_i)q_i-\kappa_i(\widehat\rho_i)
			\widehat q_i|^2 + \frac{\kappa_i^2(\widehat \rho_i)|\widehat q_i|^2}{2} 
			\bigg(-\frac{1}{\kappa_i(\rho_i)} + \frac{1}{\kappa_i(\widehat \rho_i)}
			- \frac{\kappa_i'(\widehat\rho_i)}{\kappa_i^2(\widehat\rho_i)} 
			(\rho_i-\widehat\rho_i)\bigg) \nonumber\\
    	&= \frac{1}{2\kappa_i(\rho_i)}|\kappa_i(\rho_i)q_i-\kappa_i(\widehat\rho_i)
			\widehat q_i|^2 + \frac{\kappa_i^2(\widehat \rho_i)|\widehat q_i|^2}{2}
			\bigg(-\frac{1}{\kappa_i}\bigg)(\rho_i|\widehat\rho_i).\label{eq:kq2}
    \end{align}
Assumption \ref{A4} implies that
\begin{align*}
	\bigg(-\frac{1}{\kappa_i}\bigg)(\rho_i|\widehat\rho_i)  
	&= -\frac{1}{\kappa_i(\rho_i)} + \frac{1}{\kappa_i(\widehat \rho_i)}
	-\frac{\kappa_i'(\widehat\rho_i)}{\kappa_i^2(\widehat\rho_i)} 
	|\widehat q_i|^2(\rho_i-\widehat\rho_i) \\
	&= \int_0^1\int_0^\tau -\frac{2(\kappa_i')^2-\kappa_i\kappa_i''}{\kappa_i^3}
	(s\rho_i + (1-s)\widehat\rho_i)) dsd\tau (\rho_i-\widehat\rho_i)^2 \ge 0.
\end{align*}
Due to assumption \ref{A4}, the Taylor expansion of $h_i(\rho_i|\widehat \rho_i)$ gives
\begin{align*}
	h_i(\rho_i|\widehat\rho_i) 
	&= h_i(\rho_i) - h_i(\widehat \rho_i) - h'_i(\widehat \rho_i)(\rho_i 
	- \widehat\rho_i) \\
	&= \int_0^1\int_0^\tau h''_i(s\rho_i + (1-s)\widehat\rho_i) dsd\tau 
	(\rho_i-\widehat\rho_i)^2 \ge C|\rho_i-\widehat\rho_i|^2.
\end{align*}
It follows that, for some $C>0$ independent of $\eps$,
\begin{align*}
  F_i(\rho_i,q_i|\widehat\rho_i,\widehat q_i) 
	\ge C|\rho_i-\widehat\rho_i|^2 + \frac{1}{2\kappa_i(\rho_i)}
	|\kappa_i(\rho_i)q_i-\kappa_i(\widehat\rho_i)\widehat q_i|^2 .
\end{align*}
We deduce that
$$
  \E_{\rm tot}(\brho^\eps,\bm{m}^\eps|\widehat\brho^\eps,\widehat{\bm{m}}^\eps)
	= \int_{\R^3}\sum_{i=1}^n\bigg(F_i(\rho_i^\eps,\na\rho_i^\eps
	|\widehat\rho_i^\eps,\na\widehat\rho_i^\eps)
	+ \frac12\rho_i^\eps|v_i^\eps-\widehat v_i^\eps|^2\bigg)dx \ge C\chi(t).
$$

We turn to the right-hand side of the energy inequality \eqref{3.rei}.
We write $J_1,\ldots,J_4$ for the four integrals on the right-hand side of 
\eqref{3.rei}. Thanks to assumption (A2), 
\begin{align}
  J_1 &= - \int_0^t\int_{\R^3}\sum_{i=1}^n\rho_i^\eps
	(v_i^\eps-\widehat v_i^\eps)\otimes(v_i^\eps-\widehat v_i^\eps):\na\widehat v_i^\eps
	dxds \nonumber\\  &\le C\int_0^t\int_{\R^3}\sum_{i=1}^n\rho_i^\eps|v_i^\eps
	-\widehat v_i^\eps|^2 dxds \le C\int_0^t\chi(s)ds.\label{eq:J1}
\end{align}
To estimate $J_2$, we first calculate the stress tensors using \eqref{3.S} and obtain
\begin{align*}
	s_i(\rho_i,q_i) &= p_i(\rho_i) + \frac{1}{2}(\kappa_i(\rho_i)+\rho_i\kappa_i'(\rho_i))
	|q_i|^2,\, p_i(\rho_i) =\rho_i h'_i(\rho_i)-h_i(\rho_i),  \\
	r_i(\rho_i,q_i) &= \rho_i \kappa_i(\rho_i) q_i, \\
	H_i(\rho_i,q_i) &= \kappa_i(\rho_i) q_i \otimes q_i.
\end{align*}

For $s_i(\rho_i,q_i|\widehat\rho_i,\widehat q_i)$, we first split $s_i(\rho_i,q_i)$ as 
$$
  s_i(\rho_i,q_i) = p_i(\rho_i) + \frac{1}{2}\kappa_i(\rho_i)|q_i|^2  
	+ A_i(\rho_i,q_i),\quad 
	A_i(\rho_i,q_i) = \frac{1}{2} \rho_i \kappa_i'(\rho_i) |q_i|^2.
$$
Due to assumption \ref{A4}, $p_i''$ is a continuous function. Furthermore, thanks 
to assumptions \ref{A3} and \ref{A5}, $s\rho_i + (1-s)\widehat\rho_i$ is bounded for 
$s \in [0,1]$, so $p_i''(s\rho_i + (1-s)\widehat\rho_i))$ is bounded.  
The relative pressure becomes
$$
  p_i(\rho_i|\widehat \rho_i)= \int_0^1\int_0^\tau p_i''(s\rho_i + (1-s)\widehat\rho_i) 
	dsd\tau (\rho_i-\widehat\rho_i)^2 \le C|\rho_i-\widehat\rho_i|^2.
$$
For $A_i(\rho_i,q_i|\widehat\rho_i,\widehat q_i)$, we can replace $\kappa_i(\rho_i)$ 
in the calculations of $(\frac{1}{2} \kappa_i(\rho_i) |q_i|^2) 
(\rho_i,q_i|\widehat \rho_i,\widehat q_i)$ by $\rho_i \kappa_i'(\rho_i)$ to get 
\begin{align}\label{eq:A}
	A_i(\rho_i,q_i|\widehat\rho_i,\widehat q_i)  
	= \frac{1}{2\rho_i \kappa_i'(\rho_i)}
	|\rho_i \kappa_i'(\rho_i) q_i - \widehat\rho_i \kappa_i'(\widehat\rho_i) 
	\widehat q_i|^2 
	+ \frac{|\widehat q_i|^2 \widehat{\rho_i}^2 (\kappa_i'(\widehat \rho_i))^2}{2} 
	\bigg(-\frac{1}{\rho_i \kappa_i'}\bigg)(\rho_i|\widehat\rho_i).
\end{align}
The first term on the right-hand side can be estimated as follows:
\begin{align*}
  &\frac{1}{2\rho_i \kappa_i'(\rho_i)}  |\rho_i \kappa_i'(\rho_i) q_i 
	- \widehat\rho_i \kappa_i'(\widehat\rho_i) \widehat q_i)|^2 \\
  &\phantom{xx}{}=\frac{1}{2\rho_i \kappa_i'(\rho_i)} \Big| \frac{\rho_i 
	\kappa_i'(\rho_i)}{\kappa_i(\rho_i)}(\kappa_i(\rho_i) q_i 
	- \kappa_i(\widehat \rho_i) \widehat q_i) 
	+ \bigg(\frac{\rho_i \kappa_i'(\rho_i) \kappa_i(\widehat \rho_i)}{\kappa_i(\rho_i)} 
	- \widehat\rho_i \kappa_i'(\widehat\rho_i)\bigg)\widehat q_i\Big|^2 \\
  &\phantom{xx}{}\le \frac{\rho_i \kappa_i'(\rho_i)}{2\kappa_i^2(\rho_i)} 
	|\kappa_i(\rho_i) q_i - \kappa_i(\widehat \rho_i) \widehat q_i|^2  
  + \frac{\kappa_i^2(\widehat\rho_i) \widehat q_i^2}{2\rho_i \kappa_i'(\rho_i)} 
	\bigg|\frac{\rho_i\kappa_i'(\rho_i)}{\kappa_i(\rho_i)} 
	- \frac{\widehat \rho_i \kappa_i'(\widehat\rho_i)}{\kappa_i(\widehat\rho_i)} 
	\bigg|^2 \\
  &\phantom{xx}{} \le  \frac{C}{\kappa_i(\rho_i)} |\kappa_i(\rho_i) q_i 
	- \kappa_i(\widehat \rho_i) \widehat q_i|^2 + C |\rho_i - \widehat\rho_i|^2.
\end{align*}
We use assumption \ref{A5} in the first item of the last inequality to obtain an 
upper bound on $\rho_i \kappa_i'(\rho_i)/\kappa_i(\rho_i)$. Assumptions \ref{A3} 
and \ref{A5} are used to estimate the second item.
By the same assumptions, a Taylor expansion of the last term on the right-hand side 
of \eqref{eq:A} leads to 
$$
  \bigg(-\frac{1}{\rho_i \kappa_i'}\bigg)(\rho_i|\widehat\rho_i) 
	\le C |\rho_i - \widehat \rho_i|^2.
$$
We thus have
\begin{align}\label{eq:sest}
  s_i(\rho_i,q_i|\widehat\rho_i,\widehat q_i) 
	\le C|\rho_i-\widehat\rho_i|^2 + \frac{1}{2\kappa_i(\rho_i)}|\kappa_i(\rho_i)q_i
	-\kappa_i(\widehat\rho_i)\widehat q_i|^2 .
\end{align}

Observe that
\begin{align}
	r_i&(\rho_i,q_i|\widehat \rho_i,\widehat q_i)  \nonumber \\
	&= \rho_i \kappa_i(\rho_i) q_i - \widehat \rho_i \kappa_i(\widehat \rho_i) 
	\widehat q_i - (\kappa_i(\widehat\rho_i) + \widehat \rho_i \kappa_i'(\widehat \rho_i))
	\widehat q_i (\rho_i - \widehat \rho_i) - \widehat \rho_i 
	\kappa_i(\widehat \rho_i)(q_i - \widehat q_i) \nonumber\\
	&= (\rho_i \kappa_i(\rho_i) -  \widehat \rho_i \kappa_i(\widehat \rho_i)) q_i 
	- \kappa_i(\widehat\rho_i) \widehat q_i(\rho_i - \widehat \rho_i) 
	- \widehat \rho_i \kappa_i'(\widehat \rho_i) 
	\widehat q_i (\rho_i - \widehat \rho_i) \nonumber\\
	&= \frac{\rho_i \kappa_i(\rho_i) - \widehat\rho_i \kappa_i(\widehat\rho_i)}{
	\kappa_i(\rho_i)}(\kappa_i(\rho_i)q_i - \kappa_i(\widehat\rho_i)\widehat q_i) 
	\nonumber \\
	&\phantom{xx}{} + \widehat \rho_i \kappa_i^2(\widehat\rho_i)
	\widehat q_i (-\frac{1}{\kappa_i(\rho_i)} + \frac{1}{\kappa_i(\widehat\rho_i)} 
	- \frac{\kappa_i'(\widehat\rho_i)}{\kappa_i^2(\widehat\rho_i)}
	(\rho_i-\widehat\rho_i)) \nonumber\\
	&\le \frac{C }{\kappa_i(\rho_i)}|\rho_i \kappa_i(\rho_i) 
	- \widehat \rho_i \kappa_i(\widehat \rho_i)|^2 + \frac{C }{\kappa_i(\rho_i)} 
	|\kappa_i(\rho_i)q_i - \kappa_i(\widehat\rho_i)\widehat q_i|^2 + \widehat \rho_i 
	\kappa_i^2(\widehat\rho_i)\widehat q_i (-\frac{1}{\kappa_i})(\rho_i|\widehat\rho_i)
	\nonumber\\
	&\le\frac{C }{\kappa_i(\rho_i)} |\kappa_i(\rho_i)q_i - \kappa_i(\widehat\rho_i)
	\widehat q_i|^2 +C|\rho_i - \widehat\rho_i|^2,
	\label{eq:rest}
\end{align}
where we used assumptions \ref{A3} and \ref{A5} to show the boundness of 
$(1/\kappa_i)(\rho_i)$ and $(-1/\kappa_i)(\rho_i|\widehat\rho_i)$. They are also 
used to estimate $|\rho_i \kappa_i(\rho_i) - \widehat\rho_i 
\kappa_i(\widehat\rho_i)|^2 \le C|\rho_i-\widehat\rho_i|^2$.

Next, focusing on the term $H_i(\rho_i,q_i)$, 
\begin{align}
	H_i(\rho_i,q_i|\widehat\rho_i,\widehat q_i) &=  \kappa_i(\rho_i) q_i \otimes q_i 
	- \kappa_i(\widehat\rho_i) \widehat q_i \otimes \widehat q_i 
	- \kappa_i'(\widehat\rho_i) \widehat q_i \otimes \widehat q_i (\rho_i 
	- \widehat \rho_i) \nonumber\\
	&\phantom{xx}{} - \kappa_i(\widehat\rho_i)(q_i-\widehat q_i) \otimes \widehat q_i 
	- \kappa_i(\widehat\rho_i)\widehat q_i \otimes (q_i-\widehat q_i) \nonumber\\
	&= \frac{1}{\kappa_i(\rho_i)}(\kappa_i(\rho_i) q_i - \kappa_i(\widehat \rho_i) 
	\widehat q_i) \otimes (\kappa_i(\rho_i) q_i - \kappa_i(\widehat \rho_i) 
	\widehat q_i) \nonumber\\
	&\phantom{xx}{} + \kappa_i^2(\widehat\rho_i)\widehat q_i \otimes \widehat q_i 
	\bigg(-\frac{1}{\kappa_i(\rho_i)} + \frac{1}{\kappa_i(\widehat \rho_i)}
	-\frac{\kappa_i'(\widehat\rho_i)}{\kappa_i^2(\widehat\rho_i)} 
	|\widehat q_i|^2(\rho_i-\widehat\rho_i)\bigg) \nonumber\\
	&\le \frac{1}{\kappa_i(\rho_i)} |\kappa_i(\rho_i) q_i - \kappa_i(\widehat \rho_i) 
	\widehat q_i|^2 + \kappa_i^2(\widehat \rho_i)|\widehat q_i|^2 (-\frac{1}{\kappa_i})
	(\rho_i|\widehat\rho_i)\nonumber \\
	&\le \frac{1}{\kappa_i(\rho_i)} |\kappa_i(\rho_i) q_i - \kappa_i(\widehat \rho_i) 
	\widehat q_i|^2 + C|\rho_i-\widehat \rho_i|^2.\label{eq:Hest}
\end{align} 
Combining \eqref{eq:sest}, \eqref{eq:rest}, and \eqref{eq:Hest} and using 
assumption \ref{A2}, we deduce that
\begin{align}
  J_2 &=- \int_0^t\int_{\R^3}\sum_{i=1}^n\Big(s_i(\rho_i^\eps,\na\rho_i^\eps
	|\widehat\rho_i^\eps,\na\widehat\rho_i^\eps)\diver\widehat v_i^\eps
	+ r_i(\rho_i^\eps,\na\rho_i^\eps|\widehat\rho_i^\eps,\na\widehat\rho_i^\eps)
	\cdot\na\diver\widehat v_i^\eps \nonumber \\
	&\phantom{xx}{}+H_i(\rho_i^\eps,\na\rho_i^\eps
	|\widehat\rho_i^\eps,\na\widehat\rho_i^\eps):\na\widehat v_i^\eps\Big) dxds \nonumber\\
	& \le C\int_0^t\int_{\R^3}\sum_{i=1}^n\left((\rho_i^\eps-\widehat\rho_i^\eps)^2
  +\frac{1}{\kappa_i(\rho_i^\eps)} |\kappa_i(\rho_i^\eps)\nabla\rho_i^\eps 
	- \kappa_i(\widehat\rho_i^\eps) \nabla\widehat\rho_i^\eps|^2\right)dxds	\nonumber\\
  &\le C\int_0^t\chi(s)ds. \label{eq:J2}
\end{align}

From equation \eqref{eq:rhou} we have
$$
  \widehat\rho_i^\eps\widehat u_i^\eps = -\eps\sum_{j=1}^n D_{ij}(\widehat\brho^\eps)
	\na\frac{\delta\E}{\delta\rho_j}(\widehat\brho^\eps).
$$
Hence, by definition \eqref{3.Reps} and upon using assumptions 
\ref{A3} and \ref{A1}, we see 
that $\widehat R_i^\eps$ is of order $O(\eps)$ and that
\begin{align*}
  J_3 &= -\int_0^t\int_{\R^3}\sum_{i=1}^n\frac{\rho_i^\eps}{\widehat\rho_i^\eps}
	\widehat R_i^\eps\cdot(v_i^\eps-\widehat v_i^\eps)dxds \nonumber\\
	&\le C\int_0^t\int_{\R^3}\sum_{i=1}^n\rho_i^\eps|v_i^\eps-\widehat v_i^\eps|^2 dxds
	+ C\int_0^t\int_{\R^3}\sum_{i=1}^n\rho_i^\eps\bigg(
	\frac{\widehat R_i^\eps}{\widehat\rho_i^\eps}\bigg)^2 dxds \\
  &\le C\int_0^t\int_{\R^3}\sum_{i=1}^n\rho_i^\eps|v_i^\eps-\widehat v_i^\eps|^2 dxds
	+ C\eps^2 t.
\end{align*}
Also $\widehat v_i^\eps-\widehat v_j^\eps=\widehat u_i^\eps-\widehat u_j^\eps$
is of order $\eps$, so the last term $J_4$ is estimated using assumption \ref{A5} by
\begin{align*}
  J_4 &=- \frac{1}{\eps}\int_0^t\int_{\R^3}\sum_{i,j=1}^n b_{ij}\rho_i^\eps
	(\rho_j^\eps-\widehat\rho_j^\eps)(v_i^\eps-\widehat v_i^\eps)
	\cdot(\widehat v_i^\eps-\widehat v_j^\eps)dxds\\
	 &\le C\int_0^t\int_{\R^3}\sum_{i=1}^n b_{ij}\rho_i^\eps
	|\rho_j^\eps-\widehat\rho_j^\eps||v_i^\eps-\widehat v_i^\eps|dxds \\
	&\le C\int_0^t\int_{\R^3}\sum_{i=1}^n\rho_i^\eps|v_i^\eps-\widehat v_i^\eps|^2 dxds
	+ C\int_0^t\int_{\R^3}\sum_{i,j=1}^n\rho_i^\eps|\rho_j^\eps-\widehat\rho_j^\eps|^2
	dxds \\
  &\le \int_0^t\chi(s)ds.
\end{align*}

Putting these estimates together, we arrive at
\begin{align*}
  \chi(t) &+ \frac{1}{2\eps}\int_0^t\int_{\R^3}\sum_{i,j=1}^n b_{ij}\rho_i^\eps
	\rho_j^\eps\big|(v_i^\eps-v_j^\eps)-(\widehat v_i^\eps-\widehat v_j^\eps)\big|^2
	dxds \\
	&\le C\chi(0) + C\int_0^t\chi(s)ds + C\eps^2 t.
\end{align*}
Then Gronwall's inequality gives
$\chi(t) \le C(\chi(0) + \eps^2)e ^{CT}$, finishing the proof.
\end{proof}

\begin{remark} \rm
The assumption $h''(\rho_i) \ge \alpha$ is not needed if we assume that
$\kappa_i(\rho_i){\kappa_i''(\rho_i)} -  {2\kappa_i'(\rho_i)^2} \ge \alpha$ and 
$|\nabla \widehat \rho_i|$ is bounded away from zero for any $i=1,\ldots,n$, 
because the second term on the right-hand side of \eqref{eq:kq2} controls 
$|\rho_i-\widehat \rho_i|^2$.
 
The case of quantum hydrodynamics, $\kappa_i(\rho_i) = k_i/(4\rho_i)$ is 
included in the above proof. Indeed, $\chi(t)$ is taken to be 
$$
  \chi(t) = \int_{\R^3}\sum_{i=1}^n\bigg(
	\frac12\rho_i^\eps|v_i^\eps-\widehat v_i^\eps|^2 
	+  (\rho_i^\eps-\widehat\rho_i^\eps)^2 + \frac{2\rho_i^\eps}{k_i}
	\bigg|\frac{\nabla \rho_i^\eps}{\rho_i^\eps} 
	- \frac{\nabla \widehat \rho_i^\eps}{\widehat\rho_i^\eps}\bigg|^2	\bigg)(t)dx.
$$
The condition in assumption \ref{A4} becomes
$$
\kappa_i(\rho_i){\kappa_i''(\rho_i)} -  {2\kappa_i'(\rho_i)^2} = 0,
$$
but one needs the assumption $h''_i(\rho_i)\ge \alpha$ to derive the bounds for 
$|\rho_i^\eps-\widehat\rho_i^\eps|^2$. The use of the nonlinear quadratic term 
$(2\rho_i^\eps/k_i)|\nabla \rho_i^\eps/\rho_i^\eps 
- \nabla \widehat \rho_i^\eps/\widehat\rho_i^\eps|^2$ is crucial to obtain the estimate.

Finally, for the case of constant capillarity, $\kappa_i(\rho_i)=k_i$, 
we conclude that $\kappa_i(\rho_i){\kappa_i''(\rho_i)} -  {2\kappa_i'(\rho_i)^2}=0$,
such that assumption \ref{A4} is satisfied. Thus, Theorem \ref{thm.convCE} also holds
in this case.
\end{remark}


\section{Justification of the high-friction limit}\label{sec.relax}

We recall the original system \eqref{3.rho}-\eqref{3.rhov}:
\begin{align}
  \pa_t\rho_i^\eps + \diver(\rho_i^\eps v_i^\eps) &= 0, 
	\label{4.rho} \\
	\pa_t(\rho_i^\eps v_i^\eps) + \diver(\rho_i^\eps v_i^\eps\otimes v_i^\eps)
	&= \diver S_i(\brho)
	- \frac{1}{\eps}\sum_{j=1}^n b_{ij}\rho_i^\eps\rho_j^\eps
	(v_i^\eps-v_j^\eps), \label{4.rhov}
\end{align}
where $\diver S_i = -\rho_i\na(\delta\E/\delta\rho_i)$.
The limiting system for $\eps\to 0$ becomes
\begin{align}
  \pa_t\bar\rho_i + \diver(\bar\rho_i\bar v) &= 0, \label{4.bar1} \\
  \pa_t(\bar\rho\bar v) + \diver(\bar\rho\bar v\otimes\bar v)
	&= \diver\bar S, \label{4.bar2}
\end{align}
where $\bar S=\sum_{i=1}^n S_i(\bar\rho_i)$, $\bar\rho=\sum_{i=1}^n\bar\rho_i$,
and $\bar\rho\bar v=\sum_{i=1}^n\bar\rho_i\bar v_i$. Indeed, system
\eqref{4.bar1}-\eqref{4.bar2} corresponds to the zeroth-order Chapman-Enskog
expansion \eqref{2.I01}-\eqref{2.I02}. In this section, we verify the
limit $\eps\to 0$ rigorously, analyzing the isentropic case
$F_i(\rho_i,q_i) = h_i(\rho_i)$ and the Korteweg case 
$F_i(\rho_i,q_i) = h_i(\rho_i) + \frac12\kappa_i(\rho_i)|q_i|^2$ separately.


\subsection{High-friction limit in the isentropic case}\label{sec.isen}

We consider the case when the energy density only depends on
the particle density (and not on its gradients), 
$$
  \E(\brho) = \int_{\R^3}\sum_{i=1}^n F_i(\rho_i)dx, \quad F_i=h_i(\rho_i).
$$
We prove the relaxation limit $\eps\to 0$ in \eqref{4.rho}-\eqref{4.rhov}
by applying the general result of \cite{Tza05}. Noting that
$\rho_i\na(\delta\E/\delta\rho_i) = \na p_i(\rho_i)$, where
$$
  p_i(\rho_i) = \rho_i h_i'(\rho_i) - h_i(\rho_i)
$$
is the partial pressure, we can formulate \eqref{4.rho}-\eqref{4.rhov} as
the system of balance laws
\begin{equation}\label{4.Ueps}
  \pa_t U^\eps + \diver F(U^\eps) = \frac{1}{\eps}R(U^\eps),
\end{equation}
where $U^\eps=(\brho^\eps,\bm{m}^\eps)$, 
$\bm{m}^\eps=(\rho_i^\eps v_i^\eps)_{i=1,\ldots,n}$,
\begin{align*}
  F(U^\eps) &= \begin{pmatrix} 
	\rho_i^\eps v_i^\eps \\
	\rho_i^\eps v_i^\eps\otimes v_i^\eps + p_i(\rho_i^\eps)
	\end{pmatrix}_{i=1,\ldots,n}\in\R^{2n}, \\
	R(U^\eps) &= \begin{pmatrix}
	0 \\ -\sum_{j=1}^n b_{ij}\rho_i^\eps\rho_j^\eps(v_i^\eps-v_j^\eps)
	\end{pmatrix}_{i=1,\ldots,n}\in\R^{2n}.
\end{align*}
The (formal) relaxation limit $\eps\to 0$ leads to $R(U)=0$, where 
$U=\lim_{\eps\to 0}U^\eps$. This implies that all limit velocities are the same,
$v:=v_i$ for $i=1,\ldots,n$. Thus, the limit equations are expected to be
\begin{equation*}
  \pa_t\rho_i + \diver(\rho_i v) = 0, \quad
	\pa_t(\rho v) + \diver(\rho v\otimes v) + \na p = 0,
\end{equation*}
for $i=1,\ldots,n$,
where $\rho=\sum_{i=1}^n\rho_i$ and $p=\sum_{i=1}^n p_i$. This system
can be written as the conservation law
\begin{align}\label{4.lim1}
    \pa_t u + \diver f(u) = 0,
\end{align}
where $u=(\brho,m)$, $m=\rho v$, 
and $f(u)=(\rho_1 v,\ldots,\rho_n v,\rho v\otimes v + p)$. 
System \eqref{4.Ueps} has an entropy 
$$
  \eta(U) = \sum_{i=1}^n\bigg(h_i(\rho_i) + \frac12\rho_i |v_i|^2\bigg),
$$
satisfying $\partial_t \int_{\mathbb{R}^3} \eta(U) dx \le 0$.
We introduce the relative entropy density
$$
  \eta(U^\eps|\bar U) = \sum_{i=1}^n\bigg(h_i(\rho_i^\eps|\bar\rho_i)
	+ \frac12\rho_i^\eps|v_i^\eps-\bar v|^2\bigg),
$$
where $h_i(\rho_i^\eps|\bar\rho_i)=h_i(\rho_i^\eps)-h_i(\bar\rho_i)
- h_i'(\bar\rho_i)(\rho_i^\eps-\bar\rho_i)$ and $\bar{U} 
= (\bar \rho_1, \ldots,\bar \rho_n, \bar \rho_1 \bar v, \ldots,\bar \rho_n \bar v)$.

\begin{theorem}[Relaxation limit in the isentropic case]\label{thm.isen}
Assume that \ref{N} on page \pageref{N} holds and that the function
$h_i:[0,\infty)\to\R$ is uniformly convex on $(0,\infty)$ for all $i=1,\ldots,n$.
Let $U^\eps=(\brho^\eps,\bm{v}^\eps)$ 
be a smooth solution to \eqref{4.rho}-\eqref{4.rhov} or \eqref{4.Ueps}
and let $\bar u=(\bar{\brho},\bar\rho\bar v)$ be a smooth 
solution to \eqref{4.bar1}-\eqref{4.bar2} or \eqref{4.lim1}. 
We suppose that there exists $\kappa>0$ such that
$\rho_i^\eps,\bar\rho_i\ge\kappa>0$ in $\R^3\times(0,T)$ for all $i=1,\ldots,n$. Then
for any $r>0$, there exist $s>0$ and $C>0$ independent of $\eps$ such that for 
all $t\in(0,T)$,
$$
  \int_{\{|x|<r\}}\eta(U^\eps|\bar U)(x,t)dx
	\le C\bigg(\int_{\{|x|<r+st\}}\eta(U^\eps|\bar U)(x,0)dx + \eps\bigg).
$$
In particular, if 
$$
  \lim_{\eps\to 0}\int_{\{|x|<r+st\}}\eta(U^\eps|\bar U)(x,0)dx = 0
$$
then
$$
	\lim_{\eps\to 0}\sup_{0<t<T}\int_{\R^3}\sum_{i=1}^n\big((\rho_i^\eps-\bar\rho_i)^2
	+ |v_i^\eps-\bar v_i|^2\big)dx = 0.
$$
\end{theorem}

\begin{proof}
As mentioned above, the result follows after applying Theorem 3.1 in \cite{Tza05}.
To this end, we need to verify the structural conditions (h1)-(h7) of \cite{Tza05}.

\begin{labeling}{h111}
\item[(h1)] There exists a projection matrix $\mathbb{P}:\R^{2n}\to\R^{n+1}$
satisfying $\operatorname{rank}(\mathbb{P})=n+1$ and $\mathbb{P}(R(U))=0$
for all $U\in\R^{2n}$. This matrix relates the variables $u$ and $U$ and is given by
$$
  u = \mathbb{P}U, \quad 
	\mathbb{P} = \begin{pmatrix} 
	\mathbb{I}_n & \mathbb{O}_n \\
	0,\ldots,0  & 1,\ldots,1 \end{pmatrix},
$$
where $\mathbb{I}_n$ is the unit matrix of $\R^{n\times n}$, $\mathbb{O}_n$
is the zero matrix in $\R^{n\times n}$- It holds for all $U=(\brho,\bm{m})$,
$(\mathbb{P}R(U))_i=0$ for $i=1,\ldots,n$ and
$$
  (\mathbb{P}R(U))_{n+1} = -\sum_{j,k=1}^n b_{jk}\rho_j\rho_k(v_j-v_k) = 0.
$$

\item[(h2)] The equilibrium solutions to $R(U)=0$, called $M(u)$, satisfies
$\mathbb{P}M(u)=u$. The equilibrium solutions are given by 
$M(u)=(\rho_1,\ldots,\rho_n,\rho_1 v,\ldots,\rho_n v)$, 
since $(\mathbb{P}M(u))_{i} = \rho_i$ for $i=1,\ldots,n$ and 
$(\mathbb{P}M(u))_{n+1}$ $=\sum_{j=1}^n\rho_j v=\rho v$.

\item[(h3)] The nondegeneracy conditions
$$
  \dim\operatorname{ker}(R_U(M(u))) = n+1, \quad
	\dim\operatorname{ran}(R_U(M(u))) = n-1
$$
hold, where $R_U=dR/dU$. This can be verified by a straightforward computation.

\item[(h4), (h5)] There exists an entropy density $\eta:\R^{2n}\to\R$ which
is convex and satisfies $\eta_U F_U = J_U$ and $\eta_U\cdot R(U)\le 0$,
where $J$ is the flux vector. We choose
$$
  \eta(U) = \sum_{i=1}^n\bigg(h_i(\rho_i) + \frac12\rho_i |v_i|^2\bigg), \quad
	J(U) = \sum_{i=1}^n\bigg(\rho_i h_i'(\rho_i)v_i + \frac12\rho_i|v_i|^2 v_i\bigg).
$$
Then the inequality is a consequence of the energy inequality \eqref{eq:Energy}.

\item[(h6)] The solution $u$ to \eqref{4.lim1} has the entropy-flux pair
$$
  \eta(M(u)) = \sum_{i=1}^n h_i(\rho_i) + \frac12\rho|v|^2, \quad
	J(M(u)) = \sum_{i=1}^n \rho_i h_i'(\rho_i)v + \frac12\rho|v|^2 v.
$$
This follows from \eqref{eq:Energy2} with $\widehat\rho_i,\widehat v_i$ replaced by $\bar\rho_i,\bar v$. 

\item[(h7)] 
The following inequality holds:
$$
  -\big(\eta_U(U) - \eta_U(M(u)\big)\cdot\big(R(U)-R(M(u))\big) \ge \nu|U-M(u)|^2.
$$
\end{labeling}

The inequality in (h7) amounts to proving
\begin{equation}\label{4.h7}
	\frac{1}{2}\sum_{i,j=1}^n b_{ij} \rho_i \rho_j |v_i-v_j|^2 
	\ge \nu \sum_{i=1}^n \rho_i^2|v_i-v|^2 \, .
\end{equation}
 The proof of this statement is motivated by the analysis in \cite{YYZ15}.
First, note that
$\pa\eta/\pa\rho_i=h_i'(\rho_i)-\frac12|v_i|^2$ and $\pa\eta/\pa m_i=v_i$,
where $m_i=\rho_iv_i$. Taking into account that $R(M(u))=0$, we have
\begin{align*}
  -\big(&\eta_U(U) - \eta_U(M(u))\big)\cdot\big(R(U)-R(M(u))\big) \\
  &= \sum_{i=1}^n(v_i-v)\cdot\sum_{j=1}^n b_{ij}\rho_i\rho_j(v_i-v_j) 
	=	\frac{1}{2} \sum_{i,j} b_{ij} \rho_i \rho_j |v_i-v_j|^2.
\end{align*}  

For the proof of \eqref{4.h7},
let $v_i=v+u_i$, and we reformulate the left-hand side of the inequality in (h7) as
\begin{align*}
  -\big(&\eta_U(U) - \eta_U(M(u))\big)\cdot\big(R(U)-R(M(u))\big)\\
	&= \sum_{i,j=1}^n b_{ij}\rho_i\rho_j(u_i-u_j)\cdot u_i 
	= \sum_{i,j=1}^n \tau_{ij}u_i\cdot u_j,
\end{align*}
where $\tau_{ij}=\delta_{ij}\sum_{k=1}^n b_{ik}\rho_i\rho_k - b_{ij}\rho_i\rho_j$ 
as in \eqref{eq:taudef}.
Since $(\tau_{ij})$ is not positive definite, inequality \eqref{4.h7} does
not follow directly. The idea is to use the fact that there exists a submatrix
$(\tau_{ij})\in\R^{(n-1)\times(n-1)}$ that is positive definite;
see the proof of Lemma \ref{lem.Estar}. Recalling the properties
$Q_{ij}=\delta_{ij}/\rho_i + 1/\rho_n$, $\sum_{i=1}^n\rho_iu_i=0$, 
$\sum_{i=1}^n\tau_{ij}=0$, and \eqref{eq:taucal}, we compute
\begin{align*}
  -\big(&\eta_U(U) - \eta_U(M(u))\big)\cdot\big(R(U)-R(M(u))\big) 
	= \sum_{i=1}^n u_i\sum_{j,k=1}^{n-1}\tau_{ij}Q_{jk}\rho_k u_k \\
	&= \sum_{i=1}^{n-1} u_i\sum_{j,k=1}^{n-1}	\tau_{ij}Q_{jk}\rho_k u_k 
	+ u_n\sum_{j,k=1}^{n-1}\tau_{nj}Q_{jk}\rho_k u_k \\
	&= \sum_{i=1}^{n-1} u_i\sum_{j,k=1}^{n-1}	\tau_{ij}Q_{jk}\rho_k u_k
	- \sum_{\ell=1}^{n-1}\frac{\rho_\ell u_\ell}{\rho_n}\sum_{j,k=1}^{n-1}
  \bigg(-\sum_{m=1}^{n-1}\tau_{mj}\bigg)Q_{jk}\rho_k u_k \\
	&= \sum_{i,j,k,\ell=1}^{n-1}\rho_\ell u_\ell\bigg(\frac{\delta_{i\ell}}{\rho_\ell}
	+ \frac{1}{\rho_n}\bigg)\tau_{ij}Q_{jk}\rho_k u_k \\
	&= \sum_{i,j,k,\ell=1}^{n-1}\rho_\ell u_\ell Q_{i\ell}\tau_{ij}Q_{jk}(\rho_k u_k)
	= W^TQ^\top\tau QW,
\end{align*}
where $W=(\rho_1u_1,\ldots,\rho_{n-1}u_{n-1})^\top$. 
Since $(\tau_{ij})\in\R^{(n-1)\times(n-1)}$ is
positive definite and $Q$ is invertible, $Q^\top\tau Q$ is also
positive definite. We infer that there exists a constant $\mu>0$ such that
$$
  -\big(\eta_U(U) - \eta_U(M(u))\big)\cdot\big(R(U)-R(M(u))\big)
	\ge \mu|W|^2 =  \mu\sum_{i=1}^{n-1}|\rho_i u_i|^2.
$$
We claim that we may sum from $i=1$ to $n$ using another constant.
Indeed, we infer from
$$
  |\rho_n u_n|^2 = \bigg|-\sum_{i=1}^{n-1}\rho_iu_i\bigg|^2
	\le (n-1)\sum_{i=1}^{n-1}|\rho_iu_i|^2
$$
that
$$
  \sum_{i=1}^n|\rho_iu_i|^2 
	= \sum_{i=1}^{n-1}|\rho_iu_i|^2 + |\rho_nu_n|^2
	\le n\sum_{i=1}^{n-1}|\rho_iu_i|^2
$$
and therefore,
$$
  -\big(\eta_U(U) - \eta_U(M(u))\big)\cdot\big(R(U)-R(M(u))\big)
	\ge \frac{\mu}{n}\sum_{i=1}^n|\rho_iu_i|^2,
$$
and the result follows with $\nu=\mu/n$.
\end{proof}


\subsection{High-friction limit in the Euler-Korteweg case}\label{sec.EK}

We next justify the relaxation limit $\eps\to 0$ for energies $F_i$
depending on the particle density and its gradient. We place the assumption:

\begin{enumerate}[label=\bf (A\arabic*)]
  \setcounter{enumi}{5}
	\item \label{A6} $\bar{u} = (\bar{\brho},\bar{\rho} \bar{v})$ is a smooth 
	solution to \eqref{4.bar1}-\eqref{4.bar2} satisfying
	$\bar{u}$, $\partial_t \bar{u}$, $\nabla \bar{u}$,  $D^2 \bar{u}$, $D^3 
	\bar{\rho} \in L^\infty([0,T];L^\infty(\mathbb{R}^3))$.
\end{enumerate} 

\begin{proposition}[Relative energy inequality]\label{prop.rei2}
Let $(\brho^\eps,\bm{v}^\eps)$ be a dissipative weak solution to
\eqref{4.rho}-\eqref{4.rhov} satisfying assumption \ref{A1} on page \pageref{A1} 
and let $(\bar{\brho},\bar v)$ be a smooth solution to  
\eqref{4.bar1}-\eqref{4.bar2} satisfying assumption \ref{A6}. 
Let assumption \ref{N} on page \pageref{N} hold. Then
\begin{align}
  \E_{\rm tot}&(\brho,\bm{m}|\bar\brho,\bar{\bm{m}})(t)
	+ \frac{1}{2\eps}\int_0^t\int_{\R^3}\sum_{i,j=1}^n b_{ij}\rho_i^\eps\rho_j^\eps
	|v_i^\eps-v_j^\eps|^2 dxds \nonumber \\
	&\le \E_{\rm tot}(\brho^\eps,\bm{m}^\eps|\bar\brho,\bar{\bm{m}})(0)
	- \int_0^t\int_{\R^3}\sum_{i=1}^n\rho_i^\eps(v_i^\eps-\bar v)
	\otimes(v_i^\eps-\bar v):\na\bar v dxds \nonumber \\
	&\phantom{xx}{}- \int_0^t\int_{\R^3}\sum_{i=1}^n\Big(s_i(\rho_i^\eps,\na\rho_i^\eps
	|\bar\rho_i,\na\bar\rho_i)\diver\bar v
	+ r_i(\rho_i^\eps,\na\rho_i^\eps|\bar\rho_i,\na\bar\rho_i)
	\cdot\na\diver\bar v \nonumber \\
	&\phantom{xx}{}+ H_i(\rho_i^\eps,\na\rho_i^\eps
	|\bar\rho_i,\na\bar\rho_i):\na\bar v \Big)dxds \nonumber\\
	&\phantom{xx}{}- \int_0^t\int_{\R^3}\sum_{i=1}^n \rho_i^\eps(v_i^\eps-\bar v)
	\cdot\bigg(\frac{\diver\bar S}{\bar\rho} - \frac{\diver \bar S_i}{\bar\rho_i}
	\bigg)dxds. \label{4.rei2}
\end{align}
\end{proposition}

\begin{proof}
The calculation is similar to the proof of Proposition \ref{prop.rei}. We can replace 
$\widehat\rho_i^\eps$, $\widehat v_i^\eps$ by $\bar \rho_i$, $\bar v_i$ in 
\eqref{3.rei}. To obtain the relative energy inequality, we need further to write 
the equation for $\bar\rho_i \bar v$ into the same form as \eqref{3.rhoveps2} 
and replace $\widehat R_i^\eps$ by $\bar{R}_i$, which is given by
\begin{align*}
	\partial_t(\bar\rho_i \bar v) + \diver(\bar \rho_i \bar v \otimes \bar v) 
	= - \bar{\rho}_i \nabla \frac{\delta \E}{\delta\bar \rho_i}(\bar\brho)  
	- \frac{1}{\eps} \sum_{j=1}^n b_{ij} \bar\rho_i\bar\rho_j(\bar v - \bar v) 
	+ \bar{R}_i.
\end{align*}
Using \eqref{4.bar1} and \eqref{4.bar2}, $\bar{R}_i$ can be calculated as
\begin{align*}
	\bar{R}_i &= (\partial_t \bar\rho_i + \diver(\bar\rho_i\bar v))\cdot \bar v 
	+ \bar \rho_i(\pa_t\bar v+\bar v \cdot \na\bar v) 
	+ \bar\rho_i \nabla \frac{\delta \E}{\delta\bar \rho_i}\\
	&=\frac{\bar\rho_i}{\bar \rho}(\pa_t(\bar\rho \bar v) + \na \cdot (\bar\rho 
	\bar v\otimes \bar v)) + \bar\rho_i \nabla \frac{\delta \E}{\delta\bar \rho_i} \\
	& = \frac{\bar\rho_i}{\bar\rho} \diver \bar{S} - \diver \bar{S}_i.
\end{align*}
Replacing $\widehat R_i^\eps$ with the above equation, \eqref{3.rei} becomes  
\eqref{4.rei2}. Notice that $\widehat v_i^\eps-\widehat v_j^\eps$ 
reduces to $\bar v - \bar v=0$ and the last term in \eqref{3.rei} vanishes.
\end{proof}

\begin{theorem}[Relaxation limit in the Korteweg case]\label{thm.korte}
Let $(\brho^\eps,\bm{v}^\eps)$ be a dissipative weak solution to
\eqref{4.rho}-\eqref{4.rhov} satisfying \ref{A1} on page \pageref{A1} and
\ref{A5} on page \pageref{A5} and let $(\bar\brho,\bar v)$
be a strong solution to \eqref{4.bar1}-\eqref{4.bar2} satisfying \ref{A6} on page
\pageref{A6}.
Suppose that for some constants $K>\kappa>0$, we have the uniform bounds
$\kappa\le\rho_i^\eps(x,t)\le K$ and 
$\bar\rho_i(x,t)\ge\kappa$ for all $(x,t)\in\R^3\times(0,T)$ and $i=1,\ldots,n$. 
Furthermore, let assumption \ref{N} hold.
We fix $T>0$ and set, as in Theorem \ref{thm.convCE}, 
$$
  \chi(t) = \int_{\R^3}\sum_{i=1}^n\bigg(
	\frac12\rho_i^\eps|v_i^\eps-\bar v|^2 
	+ (\rho_i^\eps-\bar\rho_i)^2+
	\frac{1}{2\kappa_i(\rho_i^\eps)}|\kappa_i(\rho_i^\eps)\nabla \rho_i^\eps 
	- \kappa_i(\bar\rho_i) \na\bar \rho_i|^2\bigg)(t)dx.
$$
Then there exists a constant $C>0$ such that for all $\eps>0$ and $t\in(0,T)$,
$$
  \chi(t) \le C(\chi(0)+\eps), \quad t\in(0,T).
$$
In particular, if $\chi(0)\to 0$ as $\eps\to 0$, we have
$$
  \sup_{t\in(0,T)}\chi(t) \to 0\quad\mbox{as }\eps\to 0.
$$
\end{theorem}

\begin{proof}
We estimate the integrals on the right-hand side of the relative entropy
inequality \eqref{4.rei2}. The second and third terms can be estimated in the same 
way as \eqref{eq:J1} and \eqref{eq:J2}, and they are bounded by $C\int_0^t \chi(s) ds$. 
We split the last term on the right-hand side of \eqref{4.rei2} into two parts:
$$
  -\int_0^t\int_{\R^3}\sum_{i=1}^n\rho_i^\eps(v_i^\eps-\bar v)\cdot
	\bigg(\frac{\diver \bar S}{\bar\rho} - \frac{\diver \bar S_i}{\bar\rho_i}\bigg)dxds
	= L_1 + L_2,
$$
where
\begin{align*}
  L_1 &= -\int_0^t\int_{\R^3}\sum_{i=1}^n \rho_i^\eps(v_i^\eps-v^\eps)
	\cdot\bigg(\frac{\diver \bar S}{\bar\rho} 
	- \frac{\diver \bar S_i}{\bar\rho_i}\bigg)dxds, \\
	L_2 &= -\int_0^t\int_{\R^3}\sum_{i=1}^n \rho_i^\eps(v^\eps-\bar v)
  \cdot\bigg(\frac{\diver \bar S}{\bar\rho} 
	- \frac{\diver \bar S_i}{\bar\rho_i}\bigg)dxds.
\end{align*} 
We infer that
\begin{align*}
  L_1 &= \int_0^t\int_{\R^3}\sum_{i=1}^n \rho_i^\eps(v_i^\eps-v^\eps)
	\frac{\diver \bar S_i}{\bar\rho_i}dxds \\
	&\le \frac{\nu}{2\eps}\int_0^t\int_{\R^3}\sum_{i=1}^n(\rho_i^\eps)^2
	|v_i^\eps-v^\eps|^2 dxds
	+ C\eps\int_0^t\int_{\R^3}\sum_{i=1}^n
	\bigg(\frac{\diver \bar S_i}{\bar\rho_i}\bigg)^2 dxds \\
  &\le \frac{\nu}{2\eps}\int_0^t\int_{\R^3}\sum_{i=1}^n (\rho_i^\eps)^2
	|v_i^\eps-v^\eps|^2 dxds + C\eps t.
\end{align*}
Using  \eqref{4.h7}, we conclude that
$$
  L_1 \le \frac{1}{4\eps}\int_0^t\int_{\R^3}\sum_{i,j=1}^n b_{ij}\rho_i^\eps\rho_j^\eps
	|v_i^\eps-v_j^\eps|^2 + C\eps t.
$$

To estimate $L_2$,  recall that $\bar S=\sum_{i=1}^n\bar S_i$
and $\rho^\eps=\sum_{i=1}^n\rho_i^\eps$, yielding
\begin{align}
  L_2 &= -\int_0^t\int_{\R^3}(v^\eps-\bar v)\cdot\sum_{i=1}^n\bigg(
	\frac{\rho^\eps}{\bar\rho}\diver \bar S_i - \frac{\rho_i^\eps}{\bar\rho_i}
	\diver\bar S_i\bigg)dxds \nonumber \\
	&= -\int_0^t\int_{\R^3}\sum_{i=1}^n\bigg(\frac{1}{\bar\rho}
	- \frac{\rho_i^\eps}{\bar\rho_i\rho^\eps}\bigg)\rho^\eps
	(v^\eps-\bar v)\cdot(\diver \bar S_i)dxds \nonumber \\
	&\le \int_0^t\int_{\R^3}\rho^\eps|v^\eps-\bar v|^2 dxds
	+ C\int_0^t\int_{\R^3}\rho^\eps\sum_{i=1}^n\bigg(\frac{1}{\bar\rho}
	- \frac{\rho_i^\eps}{\bar\rho_i\rho^\eps}\bigg)^2 dxds. \label{4.K2}
\end{align}
To estimate the first term on the right-hand side, 
we need the uniform lower and upper bounds for $\rho_i^\eps$:
\begin{align*}
  \rho^\eps|v^\eps-\bar v|^2
	&= \frac{1}{\rho^\eps}\bigg|\sum_{i=1}^n\rho_i^\eps(v_i^\eps-\bar v)\bigg|^2
  \le \frac{n}{\rho^\eps}\sum_{i=1}^n(\rho_i^\eps)^2|v_i^\eps-\bar v|^2
	\le \frac{nK}{\kappa}\sum_{i=1}^n\rho_i^\eps|v_i^\eps-\bar v|^2.
\end{align*}
The last term in \eqref{4.K2} can be estimated according to
$$
  \sum_{i=1}^n\bigg(\frac{1}{\bar\rho}
	- \frac{\rho_i^\eps}{\bar\rho_i\rho^\eps}\bigg)^2
	= \sum_{i=1}^n\bigg(\frac{\rho^\eps-\bar\rho}{\rho^\eps\bar\rho}
	+ \frac{\bar\rho_i-\rho_i^\eps}{\rho^\eps\bar\rho_i}\bigg)^2 
	\le C\sum_{i=1}^n(\rho_i^\eps-\bar\rho_i)^2.
$$
Therefore,
$$
  L_2 \le C\int_0^t\int_{\R^3}\sum_{i=1}^n\big(\rho_i^\eps|v_i^\eps-\bar v|^2
	+ (\rho_i^\eps-\bar\rho_i)^2\big)dxds \le C\int_0^t\chi(s)ds.
$$

Finally, as in the proof of Theorem \ref{thm.convCE},
$\E_{\rm tot}(\brho^\eps,\bm{m}^\eps|\bar\brho,\bar{\bm{m}})(t)\ge C\chi(t)$.
We conclude that
\begin{equation}\label{4.estchi}
  \chi(t) + \frac{1}{4\eps}\int_0^t\int_{\R^3}\sum_{i=1}^n b_{ij}\rho_i^\eps\rho_j^\eps
	|v_i^\eps-v_j^\eps|^2 dxds \le \chi(0) + C\int_0^t\chi(s)ds + C\eps t.
\end{equation}
An application of Gronwall's lemma then finishes the proof. 
\end{proof}

\begin{remark}\rm
In the previous proof, the interaction term involving $b_{ij}$ was crucial to
estimate the term $L_1$. The symmetry of $(b_{ij})$ enables us to control the
kinetic energy by the interaction energy,
$$
  \int_{\R^3}\sum_{i=1}^n\rho_i^2|v_i-v|^2dx 
	\le \frac{1}{2\nu}\int_{\R^3}\sum_{i=1}^n b_{ij}\rho_i\rho_j|v_i-v_j|^2.
$$
In the single component case, the interaction energy vanishes, and we 
recover Theorem 3 in \cite{GLT17}. 
\end{remark}


\end{document}